\documentclass[12pt]{amsart}

\usepackage{amssymb}
\usepackage{amscd}
\usepackage{amsmath,amsfonts}
\usepackage{enumerate}
\usepackage{tikz}
\usepackage{subfig}
\usepackage[all]{xy}
\usepackage{todonotes}
\usepackage[pagebackref,bookmarksopen]{hyperref}
\hypersetup{
	colorlinks=true,
	linkcolor=blue,
	citecolor=magenta,
	urlcolor=cyan,
}

\newtheorem{theo}{Theorem}[section]
\newtheorem{lemma}[theo]{Lemma}
\newtheorem{cor}[theo]{Corollary}
\newtheorem{prop}[theo]{Proposition}

\theoremstyle{definition}
\newtheorem{defi}[theo]{Definition}

\newtheorem{example}[theo]{Example}

\newtheorem{question}[theo]{Question}

\newtheorem{remark}[theo]{Remark}

\def\e{{\epsilon}}
\def\a{{\alpha}}

\newcommand{\B}{{\mathbb{B}}}
\newcommand{\C}{{\mathbb{C}}}
\newcommand{\R}{{\mathbb{R}}}
\newcommand{\Z}{{\mathbb{Z}}}
\newcommand{\K}{{\mathbb{K}}}

\newcommand{\BS}{\mathbb{S}}
\newcommand{\D}{\mathbb{D}}

\newcommand{\LL}{\mathcal{L}}
\newcommand{\set}[2]{\ensuremath{\{\,{#1}\mid {#2}\,\}}}
\DeclareMathOperator{\Ima}{Im}
\DeclareMathOperator{\Cone}{Cone}

\begin{document}

\title[Fibration theorems for non-isolated singularities]{Fibration theorems \`a la Milnor for differentiable maps with non-isolated singularities}


\author[J.~L.~Cisneros-Molina]{Jos\'e Luis Cisneros-Molina}
\address{Instituto de Matem\'aticas, Unidad Cuernavaca\\ Universidad Nacional Aut\'onoma de M\'exico\\ Avenida Universidad s/n, Colonia Lomas
de Chamilpa\\ Cuernavaca, Morelos, Mexico.}
\curraddr{}
\email{jlcisneros@im.unam.mx}
\thanks{The first author is Regular Associate of the Abdus Salam International Centre for Theoretical Physics, Trieste, Italy. Supported by CONACYT~253506.
}

\author[A.~Menegon]{Aur\'elio Menegon}
\address{Universidade Federal da Para\'iba \\ Departamento de Matem\'atica \\ CEP 58051-900, Jo\~ao Pessoa - PB, Brazil}
\curraddr{}
\email{aurelio@mat.ufpb.br}
\thanks{The second author had partial support from CNPq (Brazil).}

\author[J.~Seade]{Jos\'e Seade}
\address{Instituto de Matem\'aticas\\ Universidad Nacional Aut\'onoma de M\'exico\\ \'Area de la Investigaci\'on Cient\'ifica\\
Circuito Exterior, Ciudad Universitaria\\ Coyoac\'an, 04510, Ciudad de M\'exico, Mexico.}
\curraddr{}
\email{jseade@im.unam.mx}
\thanks{The third author was supported by CONACYT~282937 and UNAM-DGAPA-PAPIIT~IN110517}

\author[J.~Snoussi]{Jawad Snoussi}
\address{Instituto de Matem\'aticas, Unidad Cuernavaca\\ Universidad Nacional Aut\'onoma de M\'exico\\ Avenida Universidad s/n, Colonia Lomas
de Chamilpa\\ Cuernavaca, Morelos, Mexico.}
\curraddr{}
\email{jsnoussi@im.unam.mx}
\thanks{The last author was supported by UNAM-DGAPA-PAPIIT 113817}



\date{\today}

\dedicatory{}

\begin{abstract}
We prove fibration theorems \`a la Milnor for $C^\ell$ real maps with non isolated critical values. 
We study the situation for maps with linear discriminant, and prove that the concept of 
$d$-regularity is the key point for the existence of a Milnor fibration on the sphere. We also explain how one can modify the target space by homeomorphisms 
to linearize a general discriminant. Whenever the composed map is $d$-regular one has fibration on the sphere. Plenty of examples are discussed along the text, particularly the interesting family of functions $(f,g)\colon \R^n \to \R^2$ of the type 
\[
 (f,g) = \left( \sum_{i=1}^n a_i x_i^p \, , \, \sum_{i=1}^n b_i x_i^q \right) \, ,
\]
where $a_i, b_i \in \R$ are constants in generic position and $p,q \geq 2$ are integers.
\end{abstract}

\maketitle

\section{Introduction}  

We study the topological behaviour of the non-critical levels of differentible functions near a critical point.

As motivation, let $f\colon (\R^n,0) \to (\R^k,0)$ be a real analytic map, $n\geq k\geq2$, with $0 \in \R^n$ an isolated critical point and dim($f^{-1}(0)) > 0$.  By the implicit function theorem, for every sufficiently small sphere $\BS_\epsilon^{n-1}$ around $0$, one has the following \textit{transversality property}: there exists $\delta > 0$ such that for all $t\in \R^k$ with $||t|| \le \delta$ the  fiber $f^{-1}(t)$ meets $\BS_\epsilon^{n-1}$ transversally. Hence, by the Relative Ehresmann Fibration Theorem (see \cite[p.~23]{Lamotke:TCPVASL}), one has a locally trivial fibration restricting $f$ to the ``tube'':
\begin{equation}\label{eq:MLRF} 
f\colon \big( \B_\e^n \cap (f^{-1}(\BS_\delta^{k-1} \setminus \{0\}) \big)  \longrightarrow \BS_\delta^{k-1} \setminus \{0\} \,,
\end{equation}
where $\B_\e^n$ is the closed ball in $\R^n$ of radius $\e$ around the origin.
In his now classical book \cite{Milnor:SPCH} Milnor proves (see \cite[Theorem~11.2]{Milnor:SPCH} or \cite[Theorem~2]{Milnor:ISH}) that \eqref{eq:MLRF} gives rise to  a locally trivial fibration 
\begin{equation}\label{eq:MRF} 
 \phi\colon\BS^{n-1}_\e\setminus f^{-1}(0)\to \BS^{k-1} \,,
\end{equation}
where $\BS^{k-1}$ is the sphere of radius $1$ around $0 \in \R^k$. Milnor proves this constructing an integrable non-zero vector field on $\B_\e\setminus f^{-1}(0)$ which is transverse to the fibres of $f$ and to the spheres centered at $0$ contained in $\B_\e$.
The integrable curves of this vector field carry diffeomorphically the tube to the complement of $f^{-1}(\B^{k}_\delta)\cap\BS^{n-1}_\e$ in $\BS^{n-1}_\e$ keeping its boundary 
fixed.

Milnor points out in his book that this theorem has two main weaknesses: The condition for an analytic map  $(\R^n,0) \to (\R^k,0)$ with $n\geq k\geq2$, to have an isolated critical point is very stringent, and even when this is satisfied, the projection map $\phi$ may not always be taken as the natural map $f/\Vert f\Vert$ (see \cite[p.~99]{Milnor:SPCH}). 
  
The above discussion was extended in \cite[Theorem~1.3]{Pichon-Seade:FMAfgb} to functions $f$ with an isolated critical value, a more general but still stringent setting. The authors proved\footnote{It was proved in the more general context where the domain of $f$ is a real analytic variety with an isolated singularity at $0$.} that if $f$ has Thom's $a_f$ property, which implies the \textit{transversality property} mentioned above, one also has fibration \eqref{eq:MLRF} on the tube, known as Milnor-L\^e fibration, and using Milnor's vector field one also has an equivalent fibration \eqref{eq:MRF} on the sphere.

The question of whether we can take  the projection $\phi$ of \eqref{eq:MRF} to be the natural one, $\phi=f/\Vert f\Vert$ was answered in \cite{Cisneros-Seade-Snoussi:d-regular} for analytic functions with an isolated critical value, by introducing the concept of $d$-regularity (see also \cite{Cisneros-Seade-Snoussi:MFCdRAMG}): The map $f$ is $d$-regular if and only if fibration \eqref{eq:MLRF} is equivalent to a fibration \eqref{eq:MRF} where the projection is  given by $\phi=f/\Vert f\Vert$.

The $d$-regularity condition actually springs from \cite{Cisneros-Seade-Snoussi:MilnorRef} and is defined by means of a canonical pencil as follows: For every line $0 \in \ell \subset \R^k$ consider the set
\[
 X_{\ell} = \{x \in \R^n \, \vert \, f(x) \in \ell \}\,.
\]
This is a pencil of real analytic varieties intersecting at $f^{-1}(0)$ and smooth away from it.
The map $f$ is said to be \textit{$d$-regular at $0$}  if there exists $\e_0>0$ such that every $X_\ell \setminus V$ is transverse to every sphere centred at $0$ and contained in $\B_{\e_0}$, 
whenever the intersection is not empty.

Several works have been published in the last two decades studying  Milnor fibrations for real analytic maps with either an isolated critical point or value,  see for instance \cite{Ruas-Seade-Verjovsky:RSMF, Pichon:RAGOBD3S, RS,Bodin-Pichon:MFBSFL, Pichon-Seade:FMAfgb, Cisneros-Molina:Join, Cisneros-Seade-Snoussi:d-regular, Oka:NDMF, dosSantos:ERMFQHS, Bobadilla-Menegon:BM, Dutertre-dosSantos:TRMFNIS,  dosSantos-etal:FHSMG} or the survey articles \cite{Cisneros-dosSantos:MilFib, Cisneros-Seade-Snoussi:RSCG, Cisneros-Seade-Snoussi:MFCdRAMG, Seade:onfto}.

In this paper we envisage the general case of  functions  (analytic or not) with arbitrary critical set (cf. \cite{Cisneros-Grulha-Seade:OTRAM,Menegon-Seade:OLMFRAM}). Let 
$f \colon (\R^n, 0) \to  (\R^k, 0)$, $n>k \ge 2$ with a critical point at $0$, of class $C^\ell$, $\ell \ge 1$, and $\dim(f^{-1}(0)) > 0$. Our first result is (Theorem~\ref{thm:ML.f}) that $f$  has the transversality property if and only if it admits local Milnor-L\^e fibrations in tubes over the complement of the discriminant $\Delta_f$.
 
Let  $f\colon (\R^n,0) \rightarrow (\R^k,0)$ be an analytic map with  critical point at $0 \in \R^n$ and consider small balls $\B_\e^n\subset\R^n$ and $\B_\delta^k\subset\R^k$ with $0<\delta\ll\e$. Let $\Delta_\e$ be the \textit{extended discriminant} of $f$ with respect to $\B_\e^n$, which consists of the images by $f$ of critical points of $f$ in the interior of $\B_\e^n$ together with the images of the critical points
of the restriction of $f$ to the boudary sphere $\BS_\e^{n-1}$ (see Section~\ref{sec:FT}). Following \cite[\S IV.4.4]{Pham:SingInt}, in \cite[Corollary~2.2]{Cisneros-Grulha-Seade:OTRAM} was proved that the restriction of $f$ to the tube
\begin{equation}\label{eq:FT}
f\colon \B_\e\cap f^{-1}(\BS_\delta^{k-1}\setminus\Delta_\e)\to \BS_\delta^{k-1}\setminus\Delta_\e,
\end{equation}
is a locally trivial fibration, where $\BS_\delta^{k-1}=\partial\B_\delta^k$. Hence, in this general setting there is always a fibration on the tube. Moreover, if $f$ satisfies the \textit{transversality property}, 
there is no contribution to the extended discriminant $\Delta_\e$ by points on the sphere $\BS_\e^{n-1}$, the extended discriminant is just the discriminant of $f$ 
and the fibration on the tube essentially does not depend on $\e$. We remark that in \cite{Menegon-Seade:OLMFRAM}  there is a simple criterium to decide when a map has the transversality condition.
However, Milnor's vector field on $\B_\e^n\setminus f^{-1}(\Delta_\e)$ not necessarily works to inflate the tube to the sphere, since in general, $\B_\e^n\setminus f^{-1}(\Delta_\e)$ is not invariant under its flow.

In this article we extend the discussion of existence of fibrations (in the tube and in the sphere) for differentiable functions of class $\C^{\ell}$, $\ell \ge 1$, with possibly non-isolated critical value. The fibration on the tube always exists, since \cite[Corollary~2.2]{Cisneros-Grulha-Seade:OTRAM} extends without problem to this setting. Concerning the fibration on the sphere, we answer the question below:

\begin{question}\label{q:notdr}
Let  $f\colon (\R^n,0) \rightarrow (\R^k,0)$ be an analytic map with  critical point at $0 \in \R^n$ and arbitrary critical set. Is there a condition on $f$ which ensures that there exists a fibration on the sphere?
\end{question}

The answer is $d$-regularity. We show that if for a map $(\R^n,0) \buildrel {f} \over {\to} (\R^k,0)$,  $n\geq k\geq 2$, there exists a local homeomorphism $(\R^k,0) \buildrel {h} \over {\to} (\R^k,0)$ 
such that the preimage of the discriminant $h^{-1}(\Delta_f)$ is linear and the composition $f_h=h^{-1}\circ f$ is $d$-regular, then the map $f_h/\Vert f_h\Vert$ is the projection of a fibration on 
the sphere minus the preimage of $\Delta_f$ (see Theorem~\ref{prop_5b} for the precise statement).
In this case, we say that $f$ is $d_h$-regular.  Every $d$-regular map is $d_h$-regular for $h$ the identity map. In the analytic case, this fibration is equivalent to fibration \eqref{eq:FT} on the tube.

Notice that  the discriminant $\Delta_f$ can have real codimension 1 and therefore its complement splits into finitely many connected components, say $S_1,\dots,S_r$. The topology of the fibers $f^{-1}(t) $ can change for values in different $S_i$.  This happens for instance for the map germs 
\[
 (f,g) = \left( \sum_{i=1}^n a_i x_i^2 \, , \, \sum_{i=1}^n b_i x_i^2 \right) \, ,
\]
where $a_i, b_i \in \R$ are constants in generic position, studied by L\'opez de Medrano in \cite{LopezdeMedrano:SHQM}. Notice that each sector determines an analytic family 
of smooth manifolds that degenerate to the singular fiber over $0$. If we take an arc joining two adjacent sectors we get a family of smooth compact manifolds that transform 
into other smooth compact manifolds as we cross the discriminant. It would be very interesting to determine how the topology changes as we move from one sector to another. 
This can clearly be seen in the example given in Subsection~\ref{ssec:non-anal} of a non-analytic $d_h$-regular map. 

\begin{remark}\label{rem:surjective}
Throughout this paper, we will assume that $f$ is \textit{locally surjective}, that is, the image of $f$ contains an open neighbourhood of the origin in $\R^p$,
and we shall not mention it all the time. Nevertheless, it is easy to see that in the general case the same results 
hold if one intersects the bases of the locally trivial fibrations with their image. This choice is to avoid a heavy notation.
\end{remark}

\begin{remark}\label{rem:nice.germ}
The results of this article extend to analytic map-germs $f\colon (\R^n,0) \to (\R^k,0)$ if we consider $f$ in the class of \textit{nice analytic map-germs} defined in \cite[Definition~2.2]{dosSantos-etal:FHSMG}
for which the discriminant is a well-defined set-germ at $0$, so it does not depend on the radius $\e$.
\end{remark}

\vskip.2cm

The article is organized as follows. In section~\ref{sec:FT}, we give the fibration on the tube (Proposition~\ref{prop:MLF}), we define the transversality property and present Theorem~\ref{thm:ML.f} mentioned above. 
In section~\ref{section_2}, we define when $f$ has linear discriminant and we exhibit examples. Then we extend the concept of $d$-regularity and the corresponding fibration theorems to this case 
(Theorem~\ref{prop_2} and Theorem~\ref{theo_a}).  In section~\ref{section_3}, we define conic homeomorphisms on the target of $f$ and $d_h$-regularity, and we give examples of both concepts. 
Finally in section~\ref{section_5}, for an arbitrary $f$, we discuss how using a conic homeomorphism one can linearize the discriminant of $f$. 
Then one can see if the linearized map is $d$-regular, in this case we get the main fibration theorem (Theorem~\ref{prop_5b}) and we apply it to the case of 
maps with isolated critical value and the transversality property (Theorem~\ref{prop_5}). We give several examples, including the family mentioned in the abstract and an example of a non-analytic map.
We include Appendix~\ref{app:EC} where we prove Lemma~\ref{lem:FoF=F}, which says that the composition of two $C^\ell$-locally trivial fibrations between smooth manifolds is a $C^\ell$-locally trivial fibration. 
It is used in the proof of Theorem~\ref{prop_2}.

\section{Fibration on the tube}\label{sec:FT}

Let $f\colon (\R^n,0)\to (\R^k,0)$, $n >k \ge 2$, be a map of class $C^\ell$ with $\ell\geq 1$ and a critical point at $0$.
Following \cite[\S2]{Cisneros-Grulha-Seade:OTRAM} we see that there always exists a fibration on a tube.

Assume that $f$ is locally surjective (see Remark~\ref{rem:surjective}). Equip $\R^n$ with a Whitney stratification adapted to $V=f^{-1}(0)$, and let $\B_{\e_0}^n$ be a closed  ball in $\R^n$, centred at $0$, of sufficiently small radius $\e_0>0$, so that every sphere
in this ball, centred at $0$, meets transversely every stratum of $V$. We call $\e_0$ a \textit{Milnor radius} for $f$ (at $0$).
In what follows, for $0<\e<\e_0$ we shall consider the restriction $f_{\e}$ of $f$ to the closed ball $\B_{\e}^n\subset\B_{\e_0}^n$.

Denote by $\Sigma_{\e}$ the intersection of the critical set of $f$ with the ball $\B_\e$ and set $\Delta_\e := f_\e(\Sigma_{\e})$, which we call the \textit{discriminant} of $f_{\e}$.
The discriminant $\Delta_\e$ is a subanalitic set and it may depend on the choice of the radius $\e$, as showed in \cite{dosSantos-etal:FHSMG}. 

Also denote by $\Sigma_{\e}(\BS_\e^{n-1})$ the set of critical points in $\BS_\e^{n-1}$ of the restriction $f_\e|_{\BS_\e^{n-1}}$.
Set $\hat{\Sigma}_\e := \Sigma_{\e}\cup\Sigma_{\e}(\BS_\e^{n-1})$ and denote by $\hat{\Delta}_\e := f(\hat{\Sigma}_\e)$ which we call the
\textit{extended discriminant}\footnote{In \cite[\S IV.4.4]{Pham:SingInt} $\hat{\Sigma}_\e$ is called the \textit{apparent contour at the source} and $\Delta_\e$ 
is called the \textit{apparent contour at the target} or just \textit{apparent contour}.} of $f$.

Let $\B^k_\delta$ be an open ball in $\R^k$ centred at $0$ of radius $0<\delta\ll\e$.

\begin{prop}[{\cite[Corollary~2.2]{Cisneros-Grulha-Seade:OTRAM}}]\label{prop:MLF}
 Let $f\colon (\R^n,0) \to (\R^k,0)$ with $n\geq k\geq 2$ be a map of class $C^\ell$ with $\ell\geq 1$. Then the restrictions:
 \begin{align}
 f_\e|\colon \B_\e^n \cap f_\e^{-1}(\mathring{\B}_\delta^k \setminus \hat{\Delta}_\e)  &\to \mathring{\B}_\delta^k \setminus \hat{\Delta}_\e \label{eq:MST}\\
\intertext{and}
f_\e|\colon \B_\e^n \cap f_\e^{-1}(\BS_\delta^{k-1} \setminus \hat{\Delta}_\e) &\to \BS_\delta^{k-1} \setminus \hat{\Delta}_\e \label{eq:MT}
\end{align}
are locally trivial fibrations.
\end{prop}
If $f$ is analytic, then the fibrations above are smooth.

We can also consider a further condition of $f$ which ensures that the fibrations of Proposition~\ref{prop:MLF} essentially does not depend on $\e$.

\begin{defi} \label{defi_tp}
Let $0<\e<\e_0$, we say that $f$  has the \textit{transversality property in the ball $\B_\e^n$} if there exist $0<\delta\ll\e$ such that for every $y\in\B_\delta^k\setminus\Delta_\e$ 
the fibre $f^{-1}(y)$ is transverse to the sphere $\BS_\e^{n-1}$.
Since we are assuming $f$ locally surjective we can take $\B_\delta^k\subset\Ima(f_\e)$.
\end{defi}

\begin{remark}
In the case that $f$ is not locally surjective for the transversality property we need to ask that for every $y\in(\B_\delta^k\setminus\Delta_\e)\cap\Ima(f_\e)$ 
the fibre $f^{-1}(y)$ meets $\BS_{\e}^{n-1}$ transversely.
But if we consider that an empty fibre intersects transversely the sphere $\BS_{\e}^{n-1}$ we can state in general the transversality property as above.
\end{remark}

\begin{remark}
If we work with a nice analytic map-germ $f\colon(U,0) \to (\R^k,0)$ (see Remark~\ref{rem:nice.germ}), the discriminant does not depend on the radius $\e$ and we just denote it by $\Delta$. 
In this case we need the following stronger \textit{transversality property}.
We say that $f$ satisfies the \textit{strong transversality property} if for every $0<\e<\e_0$ 
there exists a real number $\delta=\delta(\e)$ such that $0<\delta\ll\e$ and for every $y\in\B_\delta\setminus\Delta$ the fibre $f^{-1}(y)$ meets $\BS_{\e}^{n-1}$ transversely.
It is clear that if $f$ satisfies the strong transversality property, then for every $0<\e<\e_0$, $f$ has the transversality property in the ball $\B_\e^n$.
\end{remark}

\begin{remark}\label{rmk:af.tp}
If $f$ is analytic and satisfies the Thom $a_f$-property then it has the strong transversality property (compare with \cite[Remark~5.7]{Cisneros-Seade-Snoussi:d-regular}).
\end{remark}

\begin{theo}\label{prop:tr.pr}
Consider $f\colon (\R^n,0)\to (\R^k,0)$, $n >k \ge 2 $ a nice map-germ with a critical point at $0$ and $\dim(f^{-1}(0))>0$. Then $f$ has the strong transversality property if and only if there exists $\e_0>0$ such that for every $0<\e<\e_0$ and every $0<\delta\ll\e$ there is no contribution to the extended discriminant $\hat{\Delta}_\e$ coming from the intersection $\BS_\e^n\cap f^{-1}(\B^k_\delta)$, that is, $\Sigma_\e=\hat{\Sigma}_\e$, and $\hat{\Delta}_\e\cap\B^k_\delta=\Delta_\e\cap\B^k_\delta=f(\Sigma_{f}(\mathring{\B}_\e))\cap\B^k_\delta$ is the discriminant of $f_\e$ restricted to $\B^k_\delta$. Therefore, fibrations \eqref{eq:MST} and \eqref{eq:MT} essentially do not depend on $\e$.
\end{theo}

So we have the following

\begin{theo}\label{thm:ML.f}
Let $f \colon (\R^n, 0) \to  (\R^k, 0)$, $n>k \ge 2$ be a map of class $C^\ell$, $\ell \ge 1$ with a critical point at $0$ and $\dim(f^{-1}(0)) > 0$. Then  the ball $\B_\e^n$ 
has the transversality property if and only if it admits local Milnor-L\^e fibrations in tubes over the complement of the discriminant $\Delta_\e$.
\end{theo}

\medskip
\section{The \texorpdfstring{$d$}{dh}-regularity for real analytic maps with linear discriminant}  \label{section_2}

\medskip
In this section, we extend the concept of $d$-regularity to real diffe\-rentiable maps with linear discriminant. 
The main idea is that $d$-regularity can be defined as transversality of preimages of lines, different from those lying in the discriminant, with sufficiently small spheres. Then we show that in this context, $d$-regularity guarantees a fibration on the sphere.

First, let us make the definition of linear discriminant  precise:

\begin{defi}
Let $f\colon (\R^n,0) \to (\R^k,0)$ be a map of class $C^\ell$, $\ell \ge 1$ with a critical point at $0$. Let $\e_0$ be a Milnor radius for $f$.
Let $0<\e<\e_0$, we say that $f$ has \textit{linear discriminant in the ball $\mathbb{B}_\e^n$} if $\Delta_\e$ is a union of line-segments with one endpoint at $0 \in \mathbb{R}^k$. 
We say that $\eta>0$ is a \textit{linearity radius} for $\Delta_\e$ if each of these line-segments intersect $\BS_\eta^{k-1}$, that is, if
\begin{align*}
\Delta_\e \cap \B_\eta^k = \Cone \big( \Delta_\e \cap \BS_\eta^{k-1} \big) \, .
\end{align*}
\end{defi}

\begin{remark}
The case when $f$ has $0\in\R^k$ as isolated critical value is considered to have linear discriminant with $\Delta_\e \cap \BS_\eta^{k-1}=\emptyset$.
\end{remark}

\begin{remark}\label{rem:delta=eta}
Let $f\colon (\B_{\e_0}^n,0) \to (\R^k,0)$ with $n\geq k\geq 2$ be a locally surjective analytic map.
Let $0<\e<\e_0$ and suppose $f$ has the transversality property in the ball $\mathbb{B}_\e^n$ and that $f_\e$ has linear discriminant in $\mathbb{B}_\e^n$.
Notice that in this case we can take the linearity radius to be the $\delta$ in the definition of the transversality property.
\end{remark}

Let $f\colon (\R^n,0) \to (\R^k,0)$ be a map of class $C^\ell$ with $\ell\geq 1$ with linear discriminant in the ball $\mathbb{B}_\e^n$ and consider a linearity radius $\eta>0$ for $f$. Set:
\[
 \mathcal{A}_\eta := \Delta_\e \cap \BS_\eta^{k-1} \, .
\]
Let $\pi\colon \BS_\eta^{k-1}\to\BS^{k-1}$ be the projection onto the unit sphere $\BS^{k-1}$ and set $\mathcal{A}=\pi(\mathcal{A}_\eta)$.
For each point $\theta \in \BS_\eta^{k-1}$, let $\mathcal{L}_\theta \subset \R^k$ be the open ray in $\R^k$ from the origin that contains the point $\theta$.
Set:
\[
 E_{\theta} := f^{-1}(\LL_\theta)  \, .
\]

Notice that $E_{\theta}$ is a manifold of class $C^\ell$ for any $\theta$ in $\BS_\eta^{k-1} \setminus \mathcal{A}_\eta$. 
 
We have:

\begin{defi}
Let $f\colon (\R^n,0) \to (\R^k,0)$ be a map of class $C^\ell$ with $\ell\geq 1$ and linear discriminant. 
We say that $f$ is \textit{$d$-regular (in the ball $\B_\e^n$)} if $E_{\theta}$ intersects the sphere $\BS_{\e'}^{n-1}$ transversely in $\R^n$, for every $\e'$ with $0<\e' \leq \e$ and for every 
$\theta \in \BS_\eta^{k-1} \setminus \mathcal{A}_\eta $.
\end{defi}

\begin{example}
Consider the real analytic map $f\colon (\R^4,0) \to (\R^3,0)$ given by:
\[
 f(x,y,z,w) := (x^2-y^2z, y, w) \, .
\]
Its critical set is the plane $\{x=y=0\}$ in $\R^4$, and its discriminant is the axis $\{u_1 = u_2 =0\}$ in $\R^3$, which is linear. One can check that $f$ is not $d$-regular. 
\end{example}

\medskip
\subsection*{A class of examples}

\begin{example} \label{ex_s}
In \cite{LopezdeMedrano:SHQM} S.~L\'opez de Medrano studied the topology of real analytic maps $(f,g): \R^n \to \R^2$ such that both $f$ and $g$ are homogeneous quadratic polynomials. For $(f,g)$ 
satisfying some generic hypothesis, he completely described the topology of $V(f,g) = f^{-1}(0) \cap g^{-1}(0)$ in terms of the coeficients of $f$ and $g$. 

Let us restrict to the diagonal case, that is, when $(f,g)$ has the form:
\[
 (f,g) = \left( \sum_{i=1}^n a_i x_i^2 \, , \, \sum_{i=1}^n b_i x_i^2 \right) \, ,
\]
where $a_i$ and $b_i$ are real constants in generic position. This means that the origin is in the convex hull of the points $(a_i,b_i)$ (which guarantees that the link of $V(f)$ is non-empty) and 
that no two of the points $(a_i,b_i)$ are linearly dependent, that is, $a_i b_j \neq a_j b_i$, for any $i \neq j$ ({\it Weak Hyperbolicity Hypothesis}).

A simple calculation shows that the set $\Sigma$ of critical points of $(f,g)$ coincides with the coordinate axis of $\R^n$ and the discriminant $\Delta(f,g)$ is the union of the $n$ line-segments in 
$\R^2$ joining the origin and the points $\lambda_i := (a_i,b_i)$. Hence $(f,g)$ has linear discriminant.

We are going to show that $(f,g)$ is $d$-regular. First, notice that any ray $\mathcal{L}_\theta$ is given by one of the following forms:

 $\{u + \a v =0\} \cap \{v > 0\} , 
 \{u + \a v =0\} \cap \{v < 0\} , 
 \{v=0\} \cap \{u > 0\} \, {\rm or} \\
 \{v=0\} \cap \{u < 0\}$
for some $\a \in \R$. Hence any $E_{\theta}$ has one of the following forms:

$\{f + \a g =0\} \cap \{g > 0\}, \,
 \{f + \a g =0\} \cap \{g < 0\}, \, 
 \{g=0\} \cap \{f > 0\} \, {\rm or }\\
 \{g=0\} \cap \{f < 0\}$
for some $\a \in \R$. 

In any case, all the equations and inequations are homogeneous, so it is easy to verify that $E_{\theta}$ intersects any sphere centered at the origin transversally.
\end{example}

\medskip
The following example generalizes L\'opez de Medrano's maps:

\begin{example} \label{ex_sg}
Let $\K$ be either $\R$ or $\C$. 
Let $(f,g): \K^n \to \K^2$ be a $\K$-analytic map of the form:
\[
 (f,g) = \left( \sum_{i=1}^n a_i x_i^p \, , \, \sum_{i=1}^n b_i x_i^p \right) \, ,
\]
where $a_i, b_i \in \K$ are constants in generic position (as in Example \ref{ex_s}) and $p \geq 2$ is an integer.

As before, one has that the critical set $\Sigma$ of $(f,g)$ is given by the coordinate axis of $\K^n$ and the discriminant $\Delta$ is linear. Moreover, the same argument of the example above shows 
that $(f,g)$ is $d$-regular.

Let us show that the link $Z$ of $V=(f,g)^{-1}(0)$ is non-empty, as before. In order to deal simultaneously with both the real and the complex case, set $c=1$ if $\K=\R$ and $c=2$ 
if $\K=\C$. 

If $p$ is odd, consider the homeomorphism:
\[
 \begin{array}{cccc}
\phi_1 \ : & \! \K^n & \! \longrightarrow & \! \K^n \\
& \! (x_1, \dots, x_n) & \! \longmapsto & \! (x_1^p, \dots, x_n^p) \\
\end{array}
\ .
\]
Then $(f,g)$ equals the composition $(f^1,g^1) \circ \phi_1$ , where $(f^1,g^1)$ is the submersion given by:
\[
 (f^1,g^1)(x) := \left( \sum_{i=1}^n a_i x_i \, , \, \sum_{i=1}^n b_i x_i \right) \, ,
\]
Hence $V$ is homeomorphic to a $(n-2)$-$\K$-dimensional plane in $\K^n$ and $Z$ is homeomorphic to the sphere $\BS^{c(n-2)-1}$.

If $p=2r$ for some odd integer $r\geq 3$, consider the homeomorphism:
\[
 \begin{array}{cccc}
\phi_2 \ : & \! \K^n & \! \longrightarrow & \! \K^n \\
& \! (x_1, \dots, x_n) & \! \longmapsto & \! (x_1^r, \dots, x_n^r) \\
\end{array}
\, .
\]
So $(f,g) = (f^2,g^2) \circ \phi_2$ , where $(f^2,g^2) := \left( \sum_{i=1}^n a_i x_i^2 \, , \, \sum_{i=1}^n b_i x_i^2 \right) $
is L\'opez de Medrano's map as in Example \ref{ex_s}. Therefore $Z$ is homeomorphic to the link of $(f^2,g^2)$, which is non-empty.

Finally, if $p=2r$ for some even integer $r\geq 2$, consider the map:
\[
 \begin{array}{cccc}
\phi_3 \ : & \! \K^n & \! \longrightarrow & \! \K^n \\
& \! (x_1, \dots, x_n) & \! \longmapsto & \! (x_1^r, \dots, x_n^r) \\
\end{array}
,
\]
which is not a homeomorphism. Once again we have that $(f,g) = (f^2,g^2) \circ \phi_3$. Hence $V = (\phi_3)^{-1}(\hat{V})$, where $\hat{V} := (f^2,g^2)^{-1}(0)$ has dimension greater than zero. So 
the dimension of $V$ is bigger than zero.
\end{example}

\medskip
Now we give a characterization of $d$-regularity for maps with linear discriminant, based on some definitions and preliminary results from 
\cite[\S3]{Cisneros-Seade-Snoussi:d-regular}.

Let $0<\e<\e_0$ be such that $f$ has linear discriminant and it is $d$-regular in the ball $\B_\e^n$.
Set $W:=f^{-1}(\Delta_\e)$ and consider the maps $\Phi\colon \B_\e^n\setminus W\to \BS^{k-1} \setminus \mathcal{A}$ and $\mathfrak{F}\colon \B_\e^n\setminus W\to\R^k\setminus\Delta_\e$ given 
respectively by
\[
 \Phi(x)=\frac{f(x)}{\Vert f(x)\Vert},\quad\text{and}\quad \mathfrak{F}=\Vert x\Vert\Phi(x).
\]

Notice that given $y\in\LL_\theta$ with $\theta\in \BS^{k-1}_\eta \setminus \mathcal{A}_\eta$ the fibre $\mathfrak{F}^{-1}(y)$ is the
intersection of $E_\theta$ with the sphere of radius $\Vert y\Vert$ centred at $0$. 
Each $E_\theta$ is a union of fibres of $\mathfrak{F}$, just as it is a union of fibres of $f$, so we have
\begin{equation}\label{eq:Et.f.F}
E_\theta=f^{-1}(\LL_\theta)=\mathfrak{F}^{-1}(\LL_\theta).
\end{equation}
The map $\mathfrak{F}$ is called the {\it  spherefication} of $f$. 

The following proposition is a straightforward generalization of \cite[Proposition~3.2]{Cisneros-Seade-Snoussi:d-regular}.

\begin{prop}\label{prop:Dreg.ch}
Let  $f\colon(\R^n,0) \to (\R^k,0)$ be a map of class $C^\ell$ with $\ell\geq 1$, with linear discriminant. Let $\e_0$ be a Milnor radius for $f$ and let $0<\e<\e_0$. 
The following conditions are equivalent
\begin{enumerate}[a)]
\item The map $f$ is $d$-regular in the ball $\B_\e^n$.\label{it:d-r} 
\item For each sphere $\BS_{\e'}^{n-1}$ in $\R^n$ centred at $0$ of radius $\e' < \e$, the restriction map $\mathfrak{F}_{\e'}\colon \BS_{\e'}^{n-1} \setminus W\to \BS^{k-1}_{\e'}\setminus\Delta_\e$ 
of $\mathfrak{F}$ is a submersion.\label{it:res-sub} 
\item The spherefication map $\mathfrak{F}$ is a submersion.\label{it:sp-sub} 
\item The map $\phi =\frac{f}{\Vert f\Vert}\colon \BS_{\e'}^{n-1} \setminus W\longrightarrow\BS^{k-1}\setminus \mathcal{A}$ is a submersion for every sphere $\BS_{\e'}^{n-1}$ with $\e'<\e$.\label{it:phi-sub}
\end{enumerate}
\end{prop}

We can now state the fibration theorem:

\begin{theo}\label{prop_2}
Let $f\colon (\R^n,0) \to (\R^k,0)$ with $n\geq k\geq 2$ be a map of class $C^\ell$ with $\ell\geq 1$. Let $\e_0$ be a Milnor radius for $f$ and let $0<\e<\e_0$.
Suppose $f$ has linear discriminant and the transversality property in the ball $\mathbb{B}_\e^n$. If $f$ is $d$-regular in the ball $\mathbb{B}_\e^n$, then 
the restriction of $\Phi$ given by
\begin{equation}\label{eq:FS}
 \phi=\Phi|\colon  \BS_\e^{n-1} \setminus W \to \BS^{k-1} \setminus \mathcal{A}
\end{equation}
is a (differentiable) locally trivial fibration over its image, where $W:= f^{-1}(\Delta_\e)$. If $f$ is analytic then the fibration $\phi$ is smooth.
\end{theo}

\begin{proof}
Suppose that $f$ is surjective. The general case is analogous. Remember from Remark~\ref{rem:delta=eta} that we can take $\delta>0$ for the transversality property
as the linearity radius of $f$.
Set
\[
 \mathcal{M} := \BS_\e^{n-1} \setminus  W  \, .
\]
Notice that $\mathcal{M}$ is an open submanifold of $\BS_\e^{n-1}$ since $\Delta_\e$ is closed in $\R^k$. Consider the following decomposition
\[
 \mathcal{M} = \left( \mathcal{M} \cap f^{-1}(\B_\delta^k) \right) \cup \left( \mathcal{M} \setminus f^{-1}(\mathring{\B}_\delta^k) \right)  \, ,
\]
where $\mathring{\B}_{\delta}^k$ is the interior of the closed ball $\B_{\delta}^k$. Both pieces are submanifolds-with-boundary of $\mathcal{M}$ of dimension $n-1$ and their intersection is the 
common boundary submanifold of dimension $n-2$
\[
\BS_\e^{n-1} \cap f^{-1}(\BS_\delta^{k-1}\setminus\Delta_\e)= \left( \mathcal{M} \cap f^{-1}(\B_\delta^k) \right) \cap \left( \mathcal{M} \setminus f^{-1}(\mathring{\B}_\delta^k) \right)\, .
\]
We are going to show that the restriction of $\phi$ to each of these components is a (differentiable) fibre bundle, so they can be glued into a global fibre bundle. 

The restriction of $f$ given by $f_1\colon\mathcal{M} \cap f^{-1}(\B_\delta^k)\to \B_{\delta}^k\setminus \Delta_\e$ is proper since $\BS_\e^{n-1}\cap f^{-1}(\B_\delta^k)$ is compact, 
and since $f$ has the transversality property in the ball $\mathbb{B}_\e^n$ it is a submersion, 
and by Ehresmann fibration theorem it is a (differentiable) fibre bundle. Now consider the radial projection $\tilde{\pi}\colon 
\B_{\delta}^k\setminus \Delta_{f,\eta} \to \BS^{k-1}\setminus\mathcal{A}$ which is a (trivial and smooth) fibre bundle. The restriction
\[
 \phi_1\colon \left( \mathcal{M} \cap f^{-1}(\B_\delta^k) \right)\to \BS^{k-1} \setminus \mathcal{A}
\]
of $\phi$ is given by the composition $\tilde{\pi}\circ f_1$. By Lemma~\ref{lem:FoF=F} the composition $\phi_1=\tilde{\pi}\circ f_1$ is a $C^\ell$-locally trivial fibration.

So now we just have to show that the restriction:
\[
 \phi_2: \mathcal{M} \setminus f^{-1}(\mathring{\B}_\delta^k) \to \BS^{k-1} \setminus \mathcal{A}
\]
is a $C^\ell$-fibration. We have that $\phi_2$ is proper since $\BS_\e^{n-1} \setminus f^{-1}(\mathring{\B}_\delta^k)$ is compact.

Since $f$ is $d$-regular, by Proposition~\ref{prop:Dreg.ch} the map $\phi\colon\BS_\e^{n-1} \setminus W\to\BS^{k-1}\setminus \mathcal{A}$ has no critical points. So $\phi_2$ is a submersion 
restricted to the interior $\mathcal{M} \setminus f^{-1}(\B_\delta^k)$ of $\mathcal{M} \setminus f^{-1}(\mathring{\B}_\delta^k)$. Since $\phi_1$ and $\phi_2$ coincide on the boundary $\mathcal{M} 
\cap f^{-1}(\BS_\delta^{k-1})$ and we already saw that $\phi_1$ restricted to this boundary is a submersion. The result follows from the Ehresmann fibration theorem for manifolds with boundary.
\end{proof}

\begin{remark}\label{rem:ICV}
 In the case of $f$ with isolated critical value and transversality property, we have $\Delta_\e=\{0\}$ and $\mathcal{A}=\emptyset$ and Theorem~\ref{prop_2} in this case gives another proof of the existence of the fibration 
on the sphere $\phi=\frac{f}{\Vert f\Vert}\colon\BS_\e^{n-1}\setminus f^{-1}(0)\to\BS^{k-1}$ without ``inflating'' the fibration \eqref{eq:MT} on the tube (see 
\cite[Theorem~5.3~(2)]{Cisneros-Seade-Snoussi:d-regular}).
\end{remark}

\medskip
\subsection{The analytic case: equivalence of fibrations}

In this subsection we consider $f\colon (\R^n,0) \to (\R^k,0)$ with $n\geq k\geq 2$, an analytic map with linear discriminant and the transversality property in the ball $\mathbb{B}_\e^n$.

Consider the locally trivial fibration \eqref{eq:MT} of Proposition~\ref{prop:MLF}. Composing it with the radial projection $\tilde{\pi}\colon \B_{\delta}^k\setminus \Delta_{f,\eta} \to \BS^{k-1}\setminus\mathcal{A}$ we get an equivalent locally trivial fibration
\[
 \tilde{f}:=\tilde{\pi}\circ f|\colon \B_{\e}^n \cap f^{-1}(\BS_\delta^{k-1} \setminus \Delta_{f,\eta}) \to \BS^{k-1} \setminus \mathcal{A}  \, .
\]

We want to show that this locally trivial fibration is equivalent to the locally trivial fibration \eqref{eq:FS}, provided that $f$ is $d$-regular. 
For this, we use the following generalization of a characterization of $d$-regularity given in
\cite[Theorem~3.7]{Cisneros-Menegon:EMFMLF} to the case of $f$ with linear discriminant. 

\begin{theo}[{\cite[Lemma~5.2]{Cisneros-Seade-Snoussi:d-regular}}]\label{lem:uni.con.str.0}  
The map $f$ is $d$-regular in the ball $\B_\e^n$, if and only if there exists  an analytic vector field $\tilde{w}$ on
 $\mathring{\B}_{\e} \setminus V$ which has the following properties:
\begin{enumerate}[(1)]\setlength{\itemsep}{0pt}
\item It is radial, {\it i.e.}, it is transverse to all spheres in $\mathring{\B}_{\e}$ centred at $0$.\label{it:pr1}
\item It is transverse to all the tubes $f^{-1}(\BS_\delta^{k-1}\setminus\Delta_\e)$.\label{it:pr3}
\item It is tangent to each $E_\theta$ for $\theta \notin\mathcal{A}_\eta$, whenever it is not empty.\label{it:pr2}
\end{enumerate}
\end{theo}

\begin{proof}
 Same as the proof of \cite[Theorem~3.7]{Cisneros-Menegon:EMFMLF} but replacing $\mathring{\B}_{\e}^n\setminus V$ by $\mathring{\B}_{\e}^n\setminus f^{-1}(\Delta_\e)$, see \cite[Remark~3.9]{Cisneros-Menegon:EMFMLF}.
\end{proof}

Now we can use the vector field $\tilde{w}$ of Theorem~\ref{lem:uni.con.str.0} to give the equivalence of fibrations.

\begin{theo}\label{theo_a}
Let $f\colon (\R^n,0) \to (\R^k,0)$ with $n\geq k\geq 2$ be a map of class $C^\ell$ with $\ell\geq 1$. Let $\e_0$ be a Milnor radius for $f$ and let $0<\e<\e_0$.
Suppose $f$ has linear discriminant, has the transversality property and that it is $d$-regular in the ball $\mathbb{B}_\e^n$. Then the fibre bundles:
\begin{gather*}
\tilde{f}:=\pi_\delta\circ f_\e|\colon \B_\e^n \cap f^{-1}(\BS_\delta^{k-1} \setminus \mathcal{A}_\delta) \to \BS^{k-1} \setminus \mathcal{A} \\
\intertext{and}
\phi\colon \BS_\e^{n-1} \setminus f^{-1}(\Delta) \to \BS^{k-1} \setminus \mathcal{A}
\end{gather*}
are equivalent, where $\pi_\delta\colon \BS_{\delta}^{k-1} \to \BS^{k-1}$ is the radial projection.
\end{theo}

\begin{proof}
In the proof of Theorem~\ref{prop_2} we saw that the restriction of $\phi$ given by  $\phi|\colon\left(\BS_\e^{n-1} \setminus W\right)\setminus f^{-1}(\mathring{\B}_\delta^k)\to \BS^{k-1} \setminus 
\mathcal{A}$ is a fibre bundle.
The flow associated to the vector field $\tilde {w}$ of Theorema~\ref{lem:uni.con.str.0}
defines in the usual way a diffeomorphism $\tau$ between $\B_\e^n \cap f^{-1}(\BS_\delta^{k-1} \setminus \Delta_\e)$ and 
$\left(\BS_\e^{n-1} \setminus W\right)\setminus f^{-1}(\mathring{\B}_\delta^k)$:
for a point $x\in \B_\e^n \cap f^{-1}(\BS_\delta^{k-1} \setminus \Delta_\e)$ follow the solution of $\tilde{w}$
that passes through $x$ till it meets $\BS_\e^{n-1}$ at some point $\hat{x}$.
This point exists and is unique because $\tilde{w}$ satisfies
conditions (\ref{it:pr1}) and (\ref{it:pr3}) in Theorema~\ref{lem:uni.con.str.0}. Define $\tau(x)=\hat{x}$.
By condition (\ref{it:pr2}) in Theorema~\ref{lem:uni.con.str.0} the solutions of $\tilde{w}$ lie in an $E_\theta$,
then we have that $\tilde{f}(x):=\Phi(x)=\Phi(\hat{x})=:\phi(\hat{x})$. Therefore, the diffeomorphism 
$\tau\colon \B_\e^n \cap f^{-1}(\BS_\delta^{k-1} \setminus \Delta_\e)\to \left(\BS_\e^{n-1} \setminus W\right)\setminus f^{-1}(\mathring{\B}_\delta^k)$ gives an equivalence of fibre bundles
\[
\xymatrix{
\B_\e^n \cap f^{-1}(\BS_\delta^{k-1} \setminus \Delta_f^\eta)\ar[rr]^{\tau}\ar[rd]_{\tilde{f}}& & \left(\BS_\e^{n-1} \setminus W\right)\setminus f^{-1}(\mathring{\B}_\delta^k)\ar[ld]^{\phi}\\
& \BS^{k-1}\setminus \mathcal{A}&
}
\]
We saw  in the proof of Theorem~\ref{prop_2} that the locally trivial fibration $\phi|\colon\left(\BS_\e^{n-1} \setminus W\right)\setminus f^{-1}(\mathring{\B}_\delta^k)\to \BS^{k-1} \setminus 
\mathcal{A}$ can be extended to the locally trivial fibration
$\phi\colon \BS_\e^{n-1} \setminus W\to \BS^{k-1} \setminus \mathcal{A}$.
\end{proof}

Now we want to apply Theorem \ref{theo_a} in Example \ref{ex_sg}. First, let us prove the following two propositions:

\begin{prop} \label{prop_surj}
Let $f: (\R^n,0) \to (\R^k,0)$ be a map of class $C^\ell$. If  $V(f) := f^{-1}(0)$ is not contained in the critical set $\Sigma_f$ of $f$, then $f$ is surjective.
\end{prop}

\begin{proof}
By hypothesis, for any $\e>0$ there exists $x \in V(f) \cap \B_\e^n$ such that $x \notin \Sigma_f$. So, up to a change of coordinates, $f$ is a projection on an open neighborhood $W_x$ of $x$ in 
$\B_\e^n$. Hence $f(W_x)$ is an open neighborhood of $0$ in $\R^k$.
\end{proof}

\begin{prop} \label{prop_tp}
Let $f: (\R^n,0) \to (\R^k,0)$ be a real analytic map. If $V(f) = \{0\}$ or if $\Sigma_f(\mathring{\B}_\e)\cap V(f)\subset \{0\}$, then $f$ has the transversality property.
\end{prop}

\begin{proof}
Suppose that $V(f) = \{0\}$ and fix $\e>0$. By continuity of $f$ there exists $\delta>0$ sufficiently small such that $f^{-1}(t) \subset \B_{\e}^n$, for any $t \in \B_\delta^k$. So $f$ has the 
transversality property.

Now suppose that $\dim V(f) >0$. Since $f$ is real analytic, there exists $\e_0>0$ sufficiently small such that $V(f)$ intersects the sphere $\BS_\e^{n-1}$ transversally in $\R^n$, for any $\e$ with 
$0<\e \leq \e_0$. 

Fix $\e$ and consider the restriction $f_|$ of $f$ to $\BS_\e^{n-1}$. Its critical points are the critical points of $f$ that are in $\BS_\e^{n-1}$ and the regular points $x$ of $f$ that are in 
$\BS_\e^{n-1}$ such that $f^{-1}(f(x))$ intersect $\BS_\e^{n-1}$ not transversally at $x$ in $\R^n$. 

Since $\Sigma_f(\mathring{\B}_\e) \cap V(f) \subset \{0\}$, it follows that any $x \in V(f) \cap \BS_\e^{n-1}$ is a regular point of $f$. Up to a change of coordinates, $f$ is a projection on an open neighborhood of 
$x$ in $\B_\e^n$, therefore the fibres of $f$ near $V$ are transverse to $\BS_\e^{n-1}$. So there exists an open neighborhood $W$ of $V(f) \cap \BS_\e^{n-1}$ in $\BS_\e^{n-1}$ such that any $y \in W$ 
is a regular point of $f_|$. So we just have to take $\delta = \delta(\e)$ sufficiently small such that $f^{-1}(\B_\delta^k) \cap \BS_\e^{n-1}$ is contained in $W$.
\end{proof}

\begin{cor}
Any map $(f,g): (\K^n,0) \to (\K^2,0)$ as in Exam\-ple \ref{ex_sg} satisfies all the hypothesis of Theorem \ref{theo_a}, hence it induces equivalent fibrations on the tube and the sphere.
\end{cor}

\begin{question}
Are fibrations \eqref{eq:MT} and \eqref{eq:FS} also equivalent when $f$, with linear  discriminant, transversality property and $d$-regularity, is only of class $C^\ell$ with $\ell\geq 1$?
\end{question}

\medskip
\section{Conic homeomorphisms and \texorpdfstring{$d_h$}{dh}-regularity}  
\label{section_3}

In this section, we want to extend the concept of $d$-regularity, allo\-wing some maps to become $d$-regular after a homeomorphism on the target space. 

Recall that given $\eta>0$, for each point $\theta \in \BS_\eta^{k-1}$ the set $\LL_\theta \subset \R^k$ is the open segment of line that starts in the origin and ends at the point $\theta$ (but not 
containing these two points).

\begin{defi}
Let $h\colon (\R^k,0) \longrightarrow (\R^k,0)$ be a homeomorphism. Suppose that there exists $\eta>0$ sufficiently small, such that the restriction of $h$ to the ball $\B_\eta^k$
\begin{equation}\label{eq:conic.homeo}
h_\eta\colon \B_\eta^k \to \mathcal{B}_\eta^k \, ,
\end{equation}
with $\mathcal{B}_\eta^k := h(\B_\eta^k)$, satisfies the following:
\begin{itemize}
\item[$(i)$] For each $\theta \in \BS_\eta^{k-1}$ the image $h_\eta(\LL_\theta)$ is a path in $\R^k$ of class $C^\ell$ with $\ell\geq 1$;
\item[$(ii)$] The inverse map $h^{-1}$ of $h$ is of class $C^\ell$  with $\ell\geq 1$ outside the origin;
\item[$(iii)$] The map $h^{-1}$ is a submersion outside the origin.
\end{itemize}
We say that the restriction \eqref{eq:conic.homeo} is a \textit{conic homeomorphism}. 
\end{defi}

The identity map is obviously a conic homeomorphism. Given a conic homemorphism $h_\eta\colon \B_\eta^k \to \mathcal{B}_\eta^k$, set:
\begin{equation*}
 \varphi_{h,\theta} := h(\LL_\theta) \, ,
\end{equation*}
for each $\theta \in \BS_\eta^{k-1}$.

\begin{example} \label{ex_2}
Consider the map $h(u,v)=(u,v^3)$ , whose inverse is given by $h^{-1}(u,v)=(u,\sqrt[3]{v})$. Clearly, $h$ is a differentiable homeomorphism, although its inverse $h^{-1}$ is not differentiable 
at the $u$-axis. It is easy to check that $h^{-1}$ is a conic homeomorphism.
\end{example}

\begin{example} \label{ex_3}
Now consider the map $h\colon (\R^2,0) \to (\R^2,0)$ given by:
\[
h(u,v)= 
\begin{cases}
(\sqrt{u}, \frac{v}{2}), &\text{if $u \geq 0$;} \\
(-\sqrt{-u}, \frac{v}{2}), &\text{if $u < 0$.}
\end{cases} 
\]
whose inverse is given by 
\[
 h^{-1}(u,v)= 
\begin{cases}
(u^2, 2v), &\text{if $u \geq 0$;} \\
(-u^2, 2v), &\text{if $u < 0$.} 
\end{cases} 
\]
One can check that $h$ is a conic homeomorphism and the inverse map $h^{-1}$ is not.
\end{example}

\medskip
If $h\colon \B_\eta^k \to \mathcal{B}_\eta^k$ is a conic homeomorphism, then each point $x \in \mathcal{B}_\eta^k$ belongs to some smooth path $\varphi_{h,\theta(x)}$. Precisely, $\theta(x) = 
\frac{h^{-1}(x)}{\|h^{-1}(x)\|}$.

So if we set
\[
 \mathcal{S}_\eta^{k-1} := h(\BS_\eta^{k-1})  \, ,
\]
the conic homeomorphism $h$ induces a continuous and surjective map
\[
\begin{array}{cccc}
\xi_h \ : & \! \mathcal{B}_\eta^k \setminus \{0\} & \! \longrightarrow & \! \mathcal{S}_\eta^{k-1} \\
& \! x & \! \longmapsto & \!  h \big( \eta \frac{h^{-1}(x)}{\|h^{-1}(x)\|} \big)
\end{array}
\, .
\]
That is, $\xi_h(x)$ is the point where the smooth curve $\varphi_{h,\theta(x)}$ that contains $x$ intersects $\mathcal{S}_\eta^{k-1}$. In other words, $\xi_h(x)$ sends each smooth curve 
$\varphi_{h,\theta}$ to the point $h(\theta) \in \mathcal{S}_\eta^{k-1}$.

So for each $\theta$ in $\BS_{\eta}^{k-1}$ we have that
\[
 h(\LL_\theta) := \varphi_{h,\theta} = \xi_h^{-1}(h(\theta)) \, .
\]

Now, given a map $f\colon (\R^n,0) \to (\R^k,0)$ and a conic homeomorphism $h: (\R^k,0) \to (\R^k,0)$, we define the map
\[
 f_h := h^{-1} \circ f \, ,
\]
which we call a {\it conic modification} of $f$:
\[
 \xymatrix { 
& (\R^n,0) \ar[d]^f \ar[dl]_{f_h} &  \\ 
(\R^k,0) \ar[r]^{h} & (\R^k,0) \ . \\
}
\]
Once a representative $h\colon \B_\eta^k \to \mathcal{B}_\eta^k$ of $h$ is fixed, for each $\theta$ in $\BS_{\eta}^{k-1}$ we set
\[
 E_{f,h,\theta} := f^{-1}(\varphi_{h,\theta}) \, .
\]

\begin{remark} \label{remark_eq}
For each $\theta \in \BS_{\eta}^{k-1}$ we have that
\[
 E_{f,h,\theta} = (f_h)^{-1}(\LL_\theta) =   \left(\eta \frac{f_h}{\| f_h \|} \right)^{-1}(\theta) = E_{f_h,id,\theta}  \, .
\]
\end{remark}

One should also notice that $E_{f,h,\theta}$ may be the empty set for some (or for any) $\theta \in \BS_{\eta}^{k-1}$, since $f$ is not necessarily surjective.

\begin{defi} 
We say that a map $f\colon (\R^n,0) \to (\R^k,0)$ and a conic homeomorphism $h: \B_\eta^k \to \mathcal{B}_\eta^k$ are \textit{compatible} in a subset $\mathcal{C} \subset \BS_{\eta}^{k-1}$ if 
$E_{f,h,\theta}$ is a differentiable manifold, for each $\theta \in \mathcal{C}$. We say that $f$ and $h$ are \textit{compatible} if $f$ and $h$ are compatible in $\BS_{\eta}^{k-1}$.
\end{defi}

\begin{remark} \label{remark_obs}
Observe that if $f$ has an isolated critical value, then any conic homeomorphism $h$ is compatible with $f$, since each $E_{f,h,\theta}$ is the inverse image by $f$ of a path $\varphi_{h,\theta}$ of 
class $C^\ell$ that contains no critical value.
\end{remark}

Now we can define $d$-regularity up to homemorphism:

\begin{defi}[$d_h$-regularity] \label{defi_dh}
Let $f\colon (\R^n,0) \to (\R^k,0)$ be a map of class $C^\ell$ with $\ell\geq 1$. Let $\e_0$ be a Milnor radius for $f$ and let $0<\e<\e_0$.
Let $h\colon \B_\eta^k \to \mathcal{B}_\eta^k$ be a conic homeomorphism compatible with $f$ in $\mathcal{C} \subset \BS_{\eta}^{k-1}$ such that $f(\B_{\e_0}^n) \subset \mathring{\mathcal{B}}_\eta^k$. 
We say that $f$ is \textit{$d_h$-regular relative to $\mathcal{C}$ in the ball $\B_\e^n$} if for any $0<\e' \leq \e$ 
the sphere $\BS_{\e'}^{n-1}$ intersects $E_{f,h,\theta}$ transversally, whenever such intersection is not empty for each $\theta\in\mathcal{C}$. 
If $\mathcal{C} = \BS_{\eta}^{k-1}$, we just say that $f$ is \textit{$d_h$-regular in the ball $\B_\e^n$}.
\end{defi}

Since by Remark \ref{remark_eq} we have that $E_{f,h,\theta} = E_{f_h,id,\theta}$, it follows:

\begin{prop} \label{lemma_4}
Let $f\colon (\R^n,0) \to (\R^k,0)$ be a $C^\ell$-map with $\ell\geq 1$ and let $h\colon \B_\eta^k \to \mathcal{B}_\eta^k$ be a conic homeomorphism. 
\begin{itemize}
\item[$(i)$] The map $f$ is compatible with $h$ in $\mathcal{C}$ if and only if the conic modification $f_h:= h^{-1} \circ f$ is compatible with the identity map in $\mathcal{C}$. 
\item[$(ii)$] Suppose that $f$ and $h$ are compatible in $\mathcal{C} \subset \BS_{\eta}^{k-1}$. Then $f$ is $d_h$-regular relative to $\mathcal{C}$  in the ball $\B_\e^n$
if and only if $f_h$ is $d$-regular relative to $\mathcal{C}$ in the ball $\B_\e^n$.
\end{itemize}
\end{prop} 

Hence the definition of $d_h$-regularity above generalizes the definition of $d$-regularity, in the sense that a $C^\ell$-map with linear discriminant (as in section 2) is 
$d$-regular if and only if it is $d_{id}$-regular relative to $\mathcal{C}:= \BS_\eta^{k-1} \setminus \mathcal{A}_\eta$.

\begin{example} \label{ex_6}
Let $f\colon (\R^3,0) \to (\R^2,0)$ be the real analytic map given by
\[
 f(x,y,z):= (x^2z+y^3, x) \, .
\]
As pointed out in Example 3.9 of \cite{Cisneros-Seade-Snoussi:MFCdRAMG}, the map $f$ has an isolated critical value, but along the line $L:= \{x-z=0 \ , \ y=0 \}$, the spaces $E_{f,id,\theta}$ 
are not transversal to the corresponding spheres. So $f$ is not $d$-regular.

Now consider the conic homeomorphisms $h(u,v) = (u,v^3)$ and its inverse $h^{-1}(u,v) = (u,\sqrt[3]{v})$ of Example \ref{ex_2}. By Remark \ref{remark_obs} we have that $f$ is compatible with both $h$ 
and $h^{-1}$.

Let us show that $f$ is $d_{(h^{-1})}$-regular: 
\[
 f_{(h^{-1})}(x,y,z) = (x^2z+y^3, x^3) \, ,
\]
it is an analytic map, whose critical set is the plane $\{x=0\}$, and its discriminant is the line $\{v=0\}$ in $\R^2$.
Using \cite[Proposition~3.8]{Cisneros-Seade-Snoussi:MFCdRAMG} one can check that $f_{(h^{-1})}$ is $d$-regular and hence $f$ is $d_{(h^{-1})}$-regular.
\end{example}

This leads us to the following: 

\begin{question}
Is there always a conic modification of a real analytic map $f$ with isolated critical value, which is $d$-regular?
\end{question}

\medskip
\section{The \texorpdfstring{$d_h$}{dh}-regularity and fibration theorems}  
\label{section_5}

In this section we study the possibility to modify the discriminant of a map by homeomorphism, making it linear. Then one can check if the new map is $d$-regular, and in that case we establish a 
fibration theorem.
This process applies in particular to maps with isolated critical value and the transversality property. 

Let $f\colon (\R^n,0) \to (\R^k,0)$, with $n \geq k \geq 2$, be a map of class $C^\ell$. Let $\e_0$ be a Milnor radius for $f$, let $0<\e<\e_0$
and let $\Delta_\e$ be the discriminant of $f_\e$.

\begin{defi} \label{defi_linea}
We say that a conic homeomorphism $h\colon \B_\eta^k \to \mathcal{B}_\eta^k$ is a linearization for $f_\e$ if we have that
\[
 h^{-1}(\Delta_\e \cap \mathcal{B}_\eta^k) = \Cone(h^{-1}(\Delta_\e \cap \partial \mathcal{B}_\eta^k)) \, .
\]
\end{defi}

Note that the definition above includes the case where $f_\e$ has an isolated critical value.

Recall the notation of Section~\ref{section_3}. Given a linearization $h$ for a map $f$ of class $C^\ell$, we have the following paths of class $C^\ell$
\[
 \varphi_{h,\theta} = h(\LL_\theta)  \, 
\]
and the induced map
\[
 \xi_{h}\colon \B_{\eta}^k \setminus \{0\} \to \BS_{\eta}^{k-1} \, ,
\]
where $\xi_{h}(x)$ is the point where the closure of the path $\varphi_{h,\theta}$ that contains $x$ intersects $\BS_{\eta}^{k-1}$. So $\varphi_{h,\theta} = (\xi_h)^{-1}(\theta)$. 

Also, given a linearization $h$ for $f_\e$, we set
\[
 \mathcal{A}_{h,\eta}:= h^{-1}(\Delta_\e \cap \partial \mathcal{B}_\eta^k) = h^{-1}(\Delta_\e) \cap \BS_\eta^{k-1} \, .
\]

Finally, for each $\theta$ in $\BS_{\eta}^{k-1}$ we have
\[
 E_{f,h,\theta} := f^{-1}(\varphi_{h,\theta}) \, ,
\]
which is a manifold of class $C^\ell$ if $\theta \notin \mathcal{A}_{h,\eta}$. Recall that
\[
 E_{f,h,\theta} = (\xi_{h} \circ f)^{-1}(\theta) = f^{-1}(h(\LL_\theta)) \, .
\]
So any linearization $h$ for a map $f_\e$ of class $C^\ell$ is compatible with $f_\e$ in $\BS_{\eta}^{k-1} \setminus \mathcal{A}_{h,\eta}$.

Once again, let $\pi\colon \BS_\eta^{k-1}\to\BS^{k-1}$ be the projection onto the unit sphere $\BS^{k-1}$ and set $\mathcal{A}_h=\pi(\mathcal{A}_{h,\eta})$. Recall the map of class $C^\ell$
\[
 \phi_{f,h}\colon \BS_\e^{n-1} \setminus K_f \to \BS^{k-1} 
\]
given by
\[
 (\phi_{f,h})(x) = \frac{f_h(x)}{\|f_h(x)\|}   \, .
\]

As before, using Proposition \ref{lemma_4} together with Theorem~\ref{prop_2} and Theo\-rem~\ref{theo_a}, we have:

\begin{theo} \label{prop_5b}
Let $f\colon (\R^n,0) \to (\R^k,0)$ be a map of class $C^\ell$ with the transversality property in the ball $\B_\e^n$, 
and suppose it admits a linearization $h\colon \B_\eta^k \to \mathcal{B}_\eta^k$ making $f$ 
$d_h$-regular in the ball $\B_\e^n$. Then the restriction
\[
 (\phi_{f,h})|\colon  \BS_\e^{n-1} \setminus f^{-1}(\Delta_f) \to \BS^{k-1} \setminus \mathcal{A}_h
\]
is a $C^\ell$-locally trivial fibration over its image. 
\end{theo}

If $f$ is analytic, by Theorem~\ref{theo_a} we have the following corollary.

\begin{cor}
In the case that $f$ is real analytic and with the transversality property in the ball $\B_\e^n$, the fibration given in Theorem~\ref{prop_5b}
is equivalent to the locally trivial fibration \eqref{eq:MT} on the tube:
\begin{equation*}
 f_h|\colon \B_\e^n \cap f^{-1}(\BS_\delta^{k-1} \setminus \Delta_\e) \to \BS_\delta^{k-1} \setminus \Delta_\e.
\end{equation*}
\end{cor}

\medskip
\subsection{Isolated critical value with transversality property}

Let $f\colon (\R^n,0) \to (\R^k,0)$ be a map of class $C^\ell$, with isolated critical value. Let $\e_0$ be a Milnor radius for $f$, let $0<\e<\e_0$ and suppose $f$ has the transversality property 
in the ball $\B_e^n$.  
In this case, one can use Milnor's vector field (\cite[Theorem~11.2]{Milnor:SPCH}) to give a fibration on the sphere, but without knowing what is the projection map. 
Since maps with linear discriminant include the case of isolated cri\-tical value, we can apply Theorem~\ref{prop_5b} to this case getting a sharper result:

\begin{theo} \label{prop_5}
Let $f\colon (\R^n,0) \to (\R^k,0)$ be a map of class $C^\ell$ with an isolated critical value and the transversality property in the ball $\B_e^n$. If it is $d_h$-regular for some conic homeomorphism $h\colon 
\B_\eta^k \to \mathcal{B}_\eta^k$ of class $C^\ell$, then the map 
\[
 \phi_{f,h} := (\pi \circ f_h)|\colon \BS_\e^{n-1} \setminus f^{-1}(0) \to \BS^{k-1} \, ,
\]
is a $C^\ell$-locally trivial fibration over its image.
\end{theo}

It is an immediate corollary of Theorem~\ref{prop_5b}, together with the fact that
the map $f_h := h^{-1} \circ f$ induces a submersion of class $C^\ell$ outside $f^{-1}(0)$.

\medskip
In the case of $f$ analytic with isolated critical value and transversality property in the ball $\B_\e^n$, we can now say more on equivalence of fibrations. By Theorem~\ref{prop_5} and Theorem~\ref{theo_a} the two fibrations:
\begin{equation} 
\phi_{f,h}\colon \BS_\e^{n-1} \setminus K_f \to \BS^{k-1} 
\end{equation}
and
\begin{equation}\label{eq:FMT.ICV}
f_h\colon \B_\e^n \cap f_h^{-1}(\BS_{\delta'}^{k-1}) \to \BS_{\delta'}^{k-1}  \end{equation}
are equivalent. This last one is clearly equivalent to the fiber bundle
\begin{equation} \label{eq:fib}
 f\colon \B_\e^n \cap f^{-1}(\mathcal{S}_{\delta'}^{k-1}) \to \mathcal{S}_{\delta'}^{k-1}  
\end{equation}

Let us show that fibration (\ref{eq:fib}) is equivalent to the fibration on the tube:
\begin{equation} \label{eq:fib2}
 f\colon \B_\e^n \cap f^{-1}(\BS_{\delta''}^{k-1}) \to \BS_{\delta''}^{k-1} 
\end{equation}
over any sphere sufficiently small such that $\BS_{\delta''}^{k-1}\subset\B_\delta^k$. This is true even when $f$ is of class $C^\ell$.

\begin{prop}\label{prop:inv.hom}
Let $f\colon (\R^n,0) \to (\R^k,0)$ be a map of class $C^\ell$ with an isolated critical value and the transversality property in the ball $\B_\e^n$.
Consider the fibration \eqref{eq:MST} on the solid tube 
\begin{equation}
 f|\colon \B_\e^n \cap f^{-1}(\mathring{\B}_\delta^k\setminus\{0\}) \to (\mathring{\B}_\delta^k\setminus\{0\}) \label{eq:SMT}.
\end{equation}
We have that $\pi_{k-1}(\mathring{\B}_\delta^k\setminus\{0\})\cong\Z$. Let $\tilde{g},\tilde{h}\colon \BS^{k-1}\to \mathring{\B}_\delta^k\setminus\{0\}$ be two continuous embeddings such that both 
represent the generator of  $\pi_{k-1}(\mathring{\B}_\delta^k\setminus\{0\})$. Then the restrictions
\begin{align}
f|\colon \B_\e^n \cap f^{-1}\bigl(g(\BS^{k-1})\bigr) &\to g(\BS^{k-1})\, ,\label{eq:fib.g}\\
\intertext{and}
f|\colon \B_\e^n \cap f^{-1}\bigl(h(\BS^{k-1})\bigr) &\to h(\BS^{k-1})\label{eq:fib.h}
\end{align}
are topologically equivalent fibre bundles.
\end{prop}

\begin{proof}
For paracompact spaces locally trivial fibrations are homotopy invariant, that is, pull-backs over homotopic maps give equivalent fibre bundles (see for instance \cite[\S4 
Theorem~9.8]{Husemoller:Bundles} plus a reduction to the principal bundle case). Since $\tilde{g}$ and $\tilde{h}$ are homotopic, the pull-backs of the fibre bundle \eqref{eq:SMT} by $\tilde{g}$ and 
$\tilde{h}$ are equivalent. Since $\tilde{g}$ and $\tilde{h}$ are embeddings those pull-backs are equivalent to the restrictions \eqref{eq:fib.g} and \eqref{eq:fib.h}.
\end{proof}

\begin{cor}
 The fibre bundles \eqref{eq:fib} and \eqref{eq:fib2} are equivalent.
\end{cor}

\begin{proof}
 Take $\tilde{g}$ as the projection of the unit sphere onto the sphere $\BS_{\delta''}^{k-1}$ and $\tilde{h}=h\circ\pi_{\delta'}$ where $\pi_{\delta'}$ is the projection of the unit sphere onto the 
sphere $\BS_{\delta''}^{k-1}$. Then $\tilde{g}$ and $\tilde{h}$ satisfy the conditions of Proposition~\ref{prop:inv.hom}.
\end{proof}

In particular, we have:

\begin{cor} \label{cor_7}
For $f$ real analytic, the fiber bundle $\phi_{f,h}$ of Proposition~\ref{prop_5} does not depend on the choice of the conic homeomorphism $h$ (as long as $f$ is $d_h$-regular), up to topological equivalence of fiber 
bundles.
\end{cor}

This means that we are not going to find different fibrations on the sphere by looking for different conic homeomorphisms $h$ such that $f$ is $d_h$-regular.

\begin{example}
Consider the real analytic map $f\colon (\R^3,0) \to (\R^2,0)$ given by
\[
 f(x,y,z):= (x+y, x^2+y^2+z^3) \, .
\]
Its critical locus $\Sigma_f$ is the line $\{x=y, z=0\}$ and its discriminant $\Delta_f$ is the parabola $\{v= \frac{u^2}{2} \}$.

Since $\Sigma_f(\mathring{\B}_\e) \cap V_f = \{0\}$, it follows from Propositions \ref{prop_surj} and \ref{prop_tp} that $f$ is surjective and has the transversality property.

The germ of conic homeomorphism $h\colon (\R^2,0) \to (\R^2,0)$ given by
\begin{align*}
 h(u,v)&= 
\begin{cases}
(\sqrt{u}, \frac{v}{2}), &\text{if $u \geq 0$,}\\
(-\sqrt{-u}, \frac{v}{2}), &\text{if $u < 0$,} \\
\end{cases} 
\intertext{whose inverse is given by }
h^{-1}(u,v)&= 
\begin{cases}
(u^2, 2v), &\text{if $u \geq 0$,}\\
(-u^2, 2v), &\text{if $u < 0$,}\\
\end{cases}
\end{align*}
is compatible with $f$ in
\[
 \mathcal{C} := \BS_\eta^1 \setminus \{ \theta_{\pi/4}, \theta_{3\pi/4} \} \, ,
\]
and $h^{-1}$ sends the parabola $\Delta_f = \{v= \frac{u^2}{2} \}$ to the set $L:= \LL_{\pi/4} \cup \LL_{3\pi/4} \cup \{0\}$. So $h$ is a linearization for $f$.

Moreover, the conic modification $f_h:= h^{-1} \circ f$,  given by
\[
 f_h(x,y,z) := 
\begin{cases}
\big( (x+y)^2  ,  2(x^2+y^2+z^3) \big),&\text{if $x+y \geq 0$,} \\
\big( -(x+y)^2  ,  2(x^2+y^2+z^3) \big),& \text{if $x+y < 0$.} 
\end{cases} 
\]
is $d$-regular relative to $\mathcal{C}$. Therefore $f$ is $d_h$-regular. 

Hence, since $f^{-1}(\Delta_f) = \{ (x-y)^2+2z^3=0 \}$, it follows by Theorem~\ref{theo_a} that the map
\[
 \phi_{f,h} = \frac{f_h}{\| f_h \|} \colon \BS_\e^2 \setminus \{ (x-y)^2+2z^3=0 \}  \to \BS_1^1 \setminus \{ \theta_{\pi/4}, \theta_{3\pi/4} \}
\]
is a topological locally trivial fibration (equivalent to the fibration in the tube associated to $f_h$).
\end{example}

\medskip
The following example generalizes Example \ref{ex_sg}:

\begin{example} \label{ex_sgg}
Let $(f,g)\colon \R^n \to \R^2$ be a real analytic map of the form
\[
 (f,g) = \left( \sum_{i=1}^n a_i x_i^p \, , \, \sum_{i=1}^n b_i x_i^q \right) \, ,
\]
with $p,q \geq 2$ integers, such that the points $\lambda_i = (a_i,b_i)$ satisfy the Weak Hyperbolicity Hypothesis (that is, no two of them are linearly dependent). Example \ref{ex_sg} concerns the 
case $p=q$. 

An easy calculation shows that the critical set $\Sigma$ of $(f,g)$ is linear and that the discriminant $\Delta$ is the union of the parametrized curves $(a_i t^p, b_i t^q)$ in $\R^2$ (with $t \in 
\R$), for each $i=1, \dots, n$. So it follows from Proposition \ref{prop_tp} that $(f,g)$ has the transversality property. 

Moreover, if $V(f,g) := (f,g)^{-1}(0)$ has dimension bigger than zero, it follows from Propositions \ref{prop_surj} that $(f,g)$ is surjective. If $V(f,g) = \{0\}$, then $(f,g)$ may be not 
surjective. Recall from Example \ref{ex_sg} that if $p=q$, then $V(f,g) = \{0\}$ if and only if the origin is not in the convex hull of the points $\lambda_i = (a_i,b_i)$.

We must consider four cases:

\begin{itemize}
\item[$(1)$] Suppose that both $p$ and $q$ are odd:

\medskip
\noindent Consider the germ of conic homeomorphism $h(u,v) = (u^{1/q}, v^{1/p})$,
whose inverse is given by $h^{-1}(u,v) = (u^q, v^p)$.
It is a linearization for the map $(f,g)$, since for each $i=1, \dots, n$ we have
\[
 h^{-1}(a_i t^p, b_i t^q) = \big( a_i^q t^{pq}, b_i^p t^{pq}  \big)  \, .
\]

\noindent Let us show that $(f,g)$ is $d_h$-regular. The conic modification $(f,g)_h := h^{-1} \circ (f,g)$ is given by
\[
 (f,g)_h = \left( \big( \sum_{i=1}^n a_i x_i^p \big)^q \, , \, \big( \sum_{i=1}^n b_i x_i^q \big)^p \right) \, .
\]
So each $E_{(f,g)_h, id, \theta}$ has one of the following forms:
\[
 \{f^q + \a g^p =0\} \cap \{g > 0\}
\]
or
\[
 \{f^q + \a g^p =0\} \cap \{g < 0\}
\]
or
\[
 \{g=0\} \cap \{f > 0\}
\]
or
\[
 \{g=0\} \cap \{f < 0\} \, ,
\]
for some $\a \in \R$. The first one corresponds to $\theta \in (0, \pi/2]$ if $\a \leq 0$ and to $\theta \in (\pi/2, \pi)$ if $\a>0$. The second one corresponds to $\theta \in (\pi, 3\pi/2]$ if $\a 
\leq 0$ and to $\theta \in (3\pi/2, 2\pi)$ if $\a>0$. The third one corresponds to $\theta = 0$ and the forth one corresponds to $\theta = \pi$.

Since all the equalities above deal with homogeneous polynomials, it is easy to see that each $E_{(f,g)_h, id, \theta}$ intersects any sphere in $\R^n$ transversally.

\medskip
\item[$(2)$] Suppose that $p$ is even and $q$ is odd:

\medskip
\noindent Consider the germ of conic homeomorphism $h: (\R^2,0) \to (\R^2,0)$ given by
\begin{align*}
h(u,v)&= 
\begin{cases}
(u^{1/q}, v^{1/p}), & \text{if $v \geq 0$,}\\
(u^{1/q}, -(-v)^{1/p}), &\text{if $v < 0$,}\\
\end{cases} 
\intertext{whose inverse is given by:}
h^{-1}(u,v)&= 
\begin{cases}
(u^q, v^p), & \text{if $v \geq 0$,}\\
(u^q, -v^p), & \text{if $v < 0$.} \\
\end{cases} 
\end{align*}
It is a linearization for $(f,g)$, since for each $i=1, \dots, n$ we have: 
\begin{align*}
h^{-1}(a_i t^p, b_i t^q) &=
\begin{cases}
\big( a_i^q t^{pq}  ,  b_i^p t^{pq}  \big), & \text{if $t \geq 0$,}\\
\big( a_i^q t^{pq}  ,  -b_i^p t^{pq}  \big), & \text{if $t < 0$,}\\
\end{cases}\quad \text{if $b_i \geq 0$,}\\
\intertext{and}
h^{-1}(a_i t^p, b_i t^q) &=
\begin{cases}
\big( a_i^q t^{pq}  ,  -b_i^p t^{pq}  \big), & \text{if $t \geq 0$,}\\
\big( a_i^q t^{pq}  ,  b_i^p t^{pq}  \big),  & \text{if $t < 0$.}\\
\end{cases}\quad\text{if $b_i < 0$.} 
\end{align*}

\noindent The conic modification $(f,g)_h $ is given by
\[
 (f,g)_h (x) = 
\begin{cases}
\big( f(x)^q, g(x)^p \big), & \text{if $g(x) \geq 0$,}\\
\big( f(x)^q, -g(x)^p \big), & \text{if $g(x) < 0$.} \\
\end{cases} 
\]

\noindent Proceeding as in Case $(1)$, one can see that $(f,g)$ is $d_h$-regular.

\medskip
\item[$(3)$] The case when $p$ is odd and $q$ is even is analogous, considering the linearization given by
\[
 h(u,v)= 
\begin{cases}
(u^{1/q}, v^{1/p}), & \text{if $u \geq 0$,}\\
(-(-u)^{1/q}, v^{1/p}), & \text{if $u < 0$.} \\
\end{cases} 
\]

\medskip
\item[$(4)$] The case when both $p$ and $q$ are even is also analogous, considering the linearization given by
\[
h(u,v)= 
\begin{cases}
(u^{1/q}, v^{1/p}), & \text{if $u \geq 0$ and  $v \geq 0$,}\\
(-(-u)^{1/q}, v^{1/p}), &\text{if $u < 0$ and $v \geq 0$,}\\
(u^{1/q}, -(-v)^{1/p}), &\text{if $u \geq 0$ and $v < 0$,} \\
(-(-u)^{1/q}, -(-v)^{1/p}), & \text{if $u < 0$ and $v < 0$,}  \\
\end{cases} 
\]
\end{itemize}
\end{example}

So we have proved:

\begin{theo}
Let $(f,g)\colon \R^n \to \R^2$ be a real analytic map of the form
\[
 (f,g) = \left( \sum_{i=1}^n a_i x_i^p \, , \, \sum_{i=1}^n b_i x_i^q \right) \, ,
\]
where $a_i, b_i \in \R$ are constants in generic position and $p,q \geq 2$ are integers, as in Example \ref{ex_sgg}. Then there exist $\e$ and $\delta$ sufficiently small, with $0<\delta \ll \e$, such 
that the restriction
\[
 (f,g)|\colon (f,g)^{-1}(\D_\delta^2 \setminus \Delta) \cap \B_\e^n \to \D_\delta^2 \setminus \Delta
\]
is a smooth locally trivial fibration over its image, where the discriminant $\Delta$ of $(f,g)$ is given by the parametrized curves $t \mapsto (a_i t^p, b_i t^q)$ in $\R^2$, for each $i=1, \dots, 
n$. 

Moreover, we have an equivalent smooth locally trivial fibration
\[
 \phi\colon \BS_\e^{n-1} \setminus \left( (f,g)^{-1}(\Delta) \cap \BS_\e^{n-1} \right) \to (\BS_{1}^1 \setminus \mathcal{A}_h) \cap \Ima(\phi)
\]
where
\[
 \phi = \frac{h^{-1} \circ f}{\|h^{-1} \circ f\|}
\]
for the corresponding homeomorphism $h$ as in Example \ref{ex_sgg}, and $\mathcal{A}_h := h^{-1}(\Delta) \cap \BS_1^1$.

If $V(f,g) \neq \{0\}$, then the maps above are surjective. 
\end{theo}

\subsection{An example of a non-analytic \texorpdfstring{$d_h$}{dh}-regular map}\label{ssec:non-anal}

Consider the real function $\varsigma\colon \R \to \R_+$ given by:
\[
 \varsigma(t) := 
\begin{cases}
e^{-1/t} & \textrm{if $t > 0$}; \\
0 & \textrm{if $t \leq 0$}. \\
\end{cases}
\]
It is a classic example of a function that is smooth and non-analytic.

\medskip
Now define $\alpha\colon\R^n\to\R$ by $\alpha(x)=1-\|x-\bar{1}\|^2$ where $\bar{1} := (1, 0, \dots, 0)$.

So the function $f: \R^n \to \R_+$ given by $f(x) := \varsigma(\alpha(x))$, is smooth and non-analytic at the origin.

\medskip
Notice that: 
\begin{itemize}
\item $f(x) = 0$ if and only if $\| x - \bar{1} \| \geq 1$;
\item $f(x) = t$ for some $t>0$ if and only if $t \leq e^{-1}$ and 
\[
 \| x - \bar{1} \|^2 = \frac{1}{\ln t} + 1.
\]
\end{itemize}

\medskip
So we have that:
\begin{enumerate}[(i)]
\item $\Ima(f) = [0, e^{-1}]$;
\item $V(f) = \R^n \setminus \mathring{\B}^n(\bar{1};1)$, where $\mathring{\B}^n(\bar{1};1)$ is the open ball of radius $1$ around the point $\bar{1}$;
\item $f^{-1}(t) = \BS^{n-1} \left( \bar{1}; \sqrt{ \frac{1}{\ln t} + 1} \right)$, where $\BS^{n-1} \left( \bar{1}; \sqrt{ \frac{1}{\ln t} + 1} \right)$ is 
the $(n-1)$-sphere around $\bar{1}$ of radius $\sqrt{ \frac{1}{\ln t} + 1}$, for any $0 < t < e^{-1}$;\label{it:fibre}
\item $f^{-1}(e^{-1}) = \{ \bar{1} \}$.
\end{enumerate}

\medskip
The gradient of $\alpha$ is given by
\[
 \nabla\alpha(x)=-2(x-\bar{1}).
\]
The derivative of $\varsigma$ is given by
\[
 \varsigma'(t)=\begin{cases}
        \frac{1}{t^2}e^{-\frac{1}{t}} & t>0\\
	0 & t\leq 0.
       \end{cases}
\]

So by the chain rule, the gradient vector of $f$ at a point $x \in \R^n$ is given by:
\[
\nabla f(x) = \varsigma'(\alpha(x)\nabla\alpha(x)=\begin{cases}
			-\frac{2}{\alpha(x)^2}f(x)(x-\bar{1}), & \text{ if $\|x-\bar{1}\|<1$,}\\
			\bar{0}, & \text{ if $\|x-\bar{1}\|\geq1$,}
         \end{cases}
\]
where $\bar{0} := (0, 0, \dots, 0)$.

Hence the critical set and the discriminant of $f$ are given by:
\[
 \Sigma_f = V(f) \cup \{ \bar{1} \},\qquad \Delta_f = \{ 0, e^{-1} \} \, .
\]

\bigskip
Now we consider a function $g$ analogous to the function $f$. Define $\beta\colon\R^n\to\R$ by $\beta(x)= 4-\|x-\bar{2}\|^2$
where $\bar{2} := (2, 0, \dots, 0)$. Define $g: \R^n \to \R_+$  by $g(x) := \varsigma(\beta(x))$.

In this case we have:
\begin{enumerate}[(1)]
\item $\Ima(g) = [0, e^{-\frac{1}{4}}]$;
\item $V(g) = \R^2 \setminus \mathring{\B}^n(\bar{2};2)$;
\item $g^{-1}(t) = \BS^{n-1} \left( \bar{2}; \sqrt{ \frac{1}{\ln t} + 4} \right)$, for any $0 < t < e^{-\frac{1}{4}}$;
\item $g^{-1}(e^{-\frac{1}{4}}) = \{ \bar{2}\}$.
\end{enumerate}
Doing a computation analogous to that for $f$, we get that the critical set and the discriminant of $g$ are given by:
\[
 \Sigma_g = V(g) \cup \{ \bar{2} \},\qquad \Delta_g = \{ 0, e^{-\frac{1}{4}} \} \, .
\]

\medskip
Finally, set the map $\Psi\colon \R^n \to \R^2$ given by $\Psi := (f,g)$. It is a smooth map that is not analytic at the origin.

\medskip
We have that:

\begin{enumerate}[(a)]
\item $\Ima(\Psi) \subset [0, e^{-1}] \times [0, e^{-\frac{1}{4}}]$;
\item $V(\Psi) = V(g) = \R^n \setminus \mathring{\B}^n(\bar{2};2)$, since $V(g) \subset V(f)$;
\item $\Psi^{-1}(t_1,t_2) = \BS^{n-1} \left( \bar{1}; \sqrt{ \frac{1}{\ln t_1} + 1} \right) \cap \BS^{n-1} \left( \bar{2}; \sqrt{ \frac{1}{\ln t_2} + 4} \right)$, for any $t_1 \neq 0$ and $t_2 \neq 
0$. Notice that if the two spheres $f^{-1}(t_1)$ and $g^{-1}(t_2)$ are transverse, the intersection is either homeomorphic to a sphere $\BS^{n-2}$ or empty. If they are tangent, the intersection is a point;\label{it:t1t2}
\item $\Psi^{-1}(0,t_2) = \left( \R^n \setminus \mathring{\B}^n(\bar{1};1) \right) \cap \BS^{n-1} \left( \bar{2}; \sqrt{ \frac{1}{\ln t_2} + 4} \right)$, for any $t_2 \neq 0$. Notice that this is 
homeomorphic to a ball $\B^{n-1}$, except when $t_2=e^{-\frac{1}{4}}$ we get the point $\{\bar{2}\}$.\label{it:0t2}
\item $\Psi^{-1}(t_1,0) = \BS^{n-1} \left( \bar{1}; \sqrt{ \frac{1}{\ln t_1} + 1} \right) \cap \left( \R^n \setminus \mathring{\B}^n(\bar{2};2) \right)$, for any $t_1 \neq 0$. Notice that this is the empty set.\label{it:t10}
\end{enumerate}

The Jacobian matrix of $\Psi$ at a point $x=(x_1, \dots, x_n)$ is given by the following matrix
\[
 \begin{bmatrix}
-\frac{2}{\alpha(x)^2} f(x) (x_1 - 1) & -\frac{2}{\alpha(x)^2} f(x) x_2 & \dots & -\frac{2}{\alpha(x)^2} f(x) x_n \\
-\frac{2}{\beta(x)^2} g(x) (x_1 - 2) & -\frac{2}{\beta(x)^2} g(x) x_2 & \dots &-\frac{2}{\beta(x)^2} g(x) x_n \\
\end{bmatrix}.
\]
\medskip
So we have that the critical set and the discriminant are given by
\[
 \Sigma_\Psi = V(f) \cup \{ x_2 = \dots = x_n =0 \},\quad \Delta_\Psi = \{ (0,t_2)\,|\, 0\leq t_2\leq e^{-1/4} \} \cup \mathcal{C} \, ,
\]
where $\mathcal{C}$ is the curve in $\R^2$ given by:
\[
 \mathcal{C}(s) := 
\begin{cases}
\left( e^{-\frac{1}{s(2-s)}} , e^{-\frac{1}{s(4-s)}} \right) & \textrm{if $s > 0$}; \\
(0,0) & \textrm{if $s \leq 0$}. \\
\end{cases}
\]
In particular, $\Psi$ \textbf{does not have} linear discriminant (see Figure~\ref{fig:discri}).

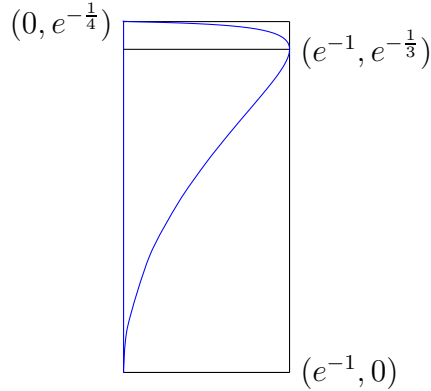
\begin{figure}[h]
\begin{tikzpicture}[scale=6,domain=0:4]
\draw[very thin,color=gray] (0,0) grid (0.3678,0.778);
\draw (0,0) -- (0.3678,0) node[right] {$(e^{-1},0)$};
\draw[color=blue] (0,0) -- (0,0.778) ;
\draw (0.3678,0) -- (0.3678,0.778);
\draw (0,0.778) node[left] {$(0,e^{-\frac{1}{4}})$} -- (0.3678,0.778);
\draw (0,0.7165)  -- (0.3678,0.7165) node[right] {$(e^{-1},e^{-\frac{1}{3}})$};
\draw[color=blue,domain=0.03:2,smooth]
plot[parametric,id=discriminant] function{exp(-1/(2*t-t*t)),exp(-1/(4*t-t*t))};
\end{tikzpicture}
\caption{The \textcolor{blue}{blue} closed curve is the discriminant.}
 \label{fig:discri}
\end{figure}


\medskip

\begin{remark}
The discriminant $\Delta_\Psi$ divides the plane in two connected components. Here we can see the phenomenon described in the Introduction: by \eqref{it:t1t2}, over points inside the discriminant, the fibres are spheres $\BS^{n-2}$, while over points outside the discriminant, the fibres are empty. 
\end{remark}

We claim that $\Psi$ has the transversality property, so there is no further contribution to the discriminant by critical points of $\Psi$ restricted to the sphere $\BS_\e^{n-1}$. Let $\B_\e$ the closed $n$-ball of small radius $\e>0$ centred at the origin $\bar{0}\in\R^n$. Consider the $(n-1)$-sphere  
$\BS^{n-1}(\bar{1};1)$ of radius $1$ centred at $\bar{1}$. Firstly, we want to find the equation of the $(n-2)$-sphere which is the intersection of the $(n-1)$-spheres $\BS^{n-1}_\e$ and 
$\BS^{n-1}(\bar{1};1)$ which respectively have the equations
\begin{align}
 x_1^2+x_2^2+\dots+x_n^2&=\e^2\label{eq:sp0}\\
 (x_1-1)^2+x_2^2+\dots+x_n^2&=1\label{eq:sp1}.
\end{align}
Getting $x_2^2$ from \eqref{eq:sp0} and substituting in \eqref{eq:sp1} we get that $x_1=\frac{\e^2}{2}$, so the intersection is the $(n-2)$-sphere with equation
\begin{equation}\label{eq:sps}
 x_2^2+\dots+x_n^2=\e^2-\frac{\e^4}{4}.
\end{equation}
Now we want to compute the radius $r$ of the $(n-1)$-sphere $\BS^{n-1}(\bar{2};r)$ with equation
\begin{equation}\label{eq:sp2}
 (x_1-2)^2+x_2^2+\dots+x_n^2=r_2^2
\end{equation}
which intersects the hyperplane $x_1=\frac{\e^2}{2}$ on the $(n-2)$-sphere given by \eqref{eq:sps}. Substituting $x_1=\frac{\e^2}{2}$ and \eqref{eq:sps} in \eqref{eq:sp2} we obtain that 
$r^2=4-\e^2$. The image of the $(n-1)$-sphere $\BS^{n-1}(\bar{2};r)$ under $g$ is $e^{-\frac{1}{r(4-r)}}$. Any $(n-1)$-sphere $\BS^{n-1}(\bar{2};r')$ of radius $r'>r>0$ intersects any $(n-1)$-sphere 
$\BS^{n-1}(\bar{1};r'')$ with $0<r''\leq1$ in either, an $(n-2)$-sphere \textbf{contained} in the interior of the $n$-ball $\B^n_\e$ or the empty set. Taking $\delta\leq e^{-\frac{1}{r(4-r)}}$ we 
get that the fibre
$\Psi^{-1}(t_1,t_2)$ with $(t_1,t_2)\in \B_\delta^k\setminus\Delta_\Psi$ is either, contained in the interior of the $n$-ball $\B^n_\e$ or the empty set, hence $\Psi$ has the transversality property.

\medskip
Consider the homeomorphism $h\colon (0,1)\times(0,1)\to(0,e^{-1})\times(0,e^{-\frac{1}{3}})$ given by 
\[
 h(u,v) := \left( e^{\frac{1}{u(u-2)}}, e^{\frac{1}{v(v-4)}} \right) \, ,
\]
with inverse
\[
 h^{-1}(u,v) := \left( 1- \sqrt{1 + \frac{1}{\ln u}} \ , \ 2- \sqrt{4 + \frac{1}{\ln v}} \right) \, .
\]
For $\eta<e^{-1}$ the restriction of $h$ to $\B_\eta^2\cap \bigl((0,1)\times(0,1)\bigr)$ is a conic homeomorphism that 
gives a linearization for $\Psi$, since $h$ takes the segment $\LL_\frac{\pi}{2}$ to itself and the segment $\LL_\frac{\pi}{4}$
onto the curve $\mathcal{C}$ (see the small rectangle in Figure~\ref{fig:discri}). 

\medskip
Set $E_\theta := (h^{-1} \circ \Psi)^{-1}(\mathcal{L}_\theta)$ for $\theta\in (0,\frac{\pi}{2}]$. For any $\theta\in(\pi/4, \pi/2)$, one can check that $E_\theta$ is a 
manifold homeomorphic to the cylinder $\BS^{n-2} \times (0,1)$ that intersects the sphere $\BS_{\e'}$ transversally, for any $\e' \leq \e$, with $\e$ small enough as above, and for $\theta\in[0,\pi/4)$ we have that $E_\theta$ is empty. Moreover: 
\[
 E_{\frac{\pi}{4}} = \{ x_2 = \dots = x_n =0 \}
\]
and $E_{\frac{\pi}{2}}$ is a manifold homeomorphic to a disk $\D^n$ that intersects the sphere $\BS_{\e'}$ transversally, for any $\e' \leq \e$. Hence $\Psi$ is $d_h$-regular.
Alternatively, one can check this by using Proposition 3.8 of [3] for the composition $h^{-1} \circ \Psi$.

\appendix

\section{Proof of Lemma~\ref{lem:FoF=F}}\label{app:EC}

In this section we use an extension of Ehresmann Fibration Theorem proved by Wolf in \cite{Wolf:DFEMCRM} to prove the following:

\begin{lemma}\label{lem:FoF=F}
Let $f\colon X\to Y$ and $g\colon Y\to Z$ be $C^\ell$-locally trivial fibrations with $1\leq l\leq\infty$, between smooth manifolds possibly with boundary. Then $g\circ f\colon X\to Z$ is a 
$C^\ell$-locally trivial fibration.
\end{lemma}

Analogous results are given by Ekedahl \cite{Ekedahl:CBPBP} and McKay \cite[Corollary~7]{Mckay:MCC}.

We follow Section~2 of \cite{Wolf:DFEMCRM} to give the necessary definitions to enunciate Wolf's theorem. In \cite{Wolf:DFEMCRM} the results are stated for smooth manifolds and
smooth maps between them. Here we also deal with smooth ($C^\infty$) manifolds but the maps may be only of class $C^\ell$ for $1\leq l\leq\infty$.

Let $\varphi\colon E\to B$ be a submersion of class $C^\ell$ with $1\leq l\leq\infty$. Since $\varphi$ is a submersion, for any $b\in B$ the fibre
$\varphi^{-1}(b)$ is a submanifold of $E$ of dimension $\dim E-\dim B$.

Given $x\in E$, the \textit{vertical space} $V_x$ at $x$
is the subspace of $T_xE$ defined by
\[
 V_x=\set{v\in T_xE}{D_x\varphi(v)=0},
\]
that is, the  space tangent to the fibre
$\varphi^{-1}(\varphi(x))$. One has that $\dim V_x=\dim E-\dim B$.
The \textit{vertical distribution} is $\mathcal{V}=\{V_x\}_{x\in
E}$. An \textit{Ehresmann connection} for $\varphi$ is a
smooth distribution $\mathcal{H}=\{H_x\}_{x\in E}$ on $E$
that is complementary to $\mathcal{V}$, {\it i.e.},
$T_xE=V_x\oplus H_x$ for every $x\in E$. So $D_x\varphi$ restricts
to a linear isomorphism from $H_x$ onto $T_{\varphi(x)}B$. The
space $H_x$ is the \textit{horizontal space} at $x$. Notice that
using a Riemannian metric on $E$ it is always possible to
construct an Ehresmann connection taking the orthogonal complement
of the vertical distribution.

Fix an Ehresmann connection $\mathcal{H}$ of $\varphi\colon E\to
B$. A tangent vector $v\in T_x E$ is \textit{horizontal}
(respectively, \textit{vertical}) if $v\in H_x$ (respectively,
$v\in V_x$); a sectionally $C^\ell$-curve in $E$ with $1\leq l\leq\infty$ is
\textit{horizontal} (respectively, \textit{vertical}) if each of
its tangent vectors is \textit{horizontal} (respectively,
\textit{vertical}). We make the convention that all sectionally
$C^\ell$-curves are parametrised so as to be regular (nowhere
vanishing tangent vector) on each smooth arc.

Let $\alpha(t)$, $t\in[0,1]$, be a sectionally $C^\ell$-curve in
$\varphi(E)\subset B$. Given $x\in\varphi^{-1}(\alpha(0))$, there
is at most one sectionally $C^\ell$-\textit{horizontal curve}
$\alpha_x(t)$, $t\in[0,1]$, in $E$ such that:
\begin{enumerate}[(a)]
 \item $\alpha_x(0)=x$, and
\item $\varphi\circ\alpha_x=\alpha$.
\end{enumerate}
If it exists, $\alpha_x$ is called the \textit{horizontal lift} of
$\alpha$ to $x$. If $\alpha_x$ exists for every
$x\in\varphi^{-1}(\alpha(0))$, then we say that $\alpha$
\textit{has horizontal lifts}. In such case, the
\textit{translation of the fibres along $\alpha$} is the map
\begin{align*}
 \rho_\alpha\colon\varphi^{-1}(\alpha(0))&\to \varphi^{-1}(\alpha(1)),\\
x&\mapsto\alpha_x(1),
\end{align*}
which is differentiable by Lemma~2.2 of \cite{Wolf:DFEMCRM}.

\begin{theo}[{Corollary~2.5 of \cite{Wolf:DFEMCRM}}]\label{thm:EEFT}
Let $\varphi\colon E\to B$ be a submersion of class $C^\ell$ with $1\leq l\leq\infty$, where $E$ and $B$ are paracompact and $B$
connected.\footnote{The hypothesis in \cite{Wolf:DFEMCRM} of $E$ being
connected is not used in the proof and it works without it.} Then
the following statements are equivalent:
\begin{enumerate}[(i)]
 \item $\varphi\colon E\to\varphi(E)$ is a locally trivial differentiable fibre
 bundle.
\item There exists an Ehresmann connection for $\varphi$, relative
to which every sectionally $C^\ell$-curve in $\varphi(E)$ has
horizontal lifts. \item If $\mathcal{H}$ is an Ehresmann
connection for $\varphi$, then every sectionally $C^\ell$-curve in
$\varphi(E)$ has horizontal lifts relative to $\mathcal{H}$.
\end{enumerate}
\end{theo}

It is easy to extend Theorem~\ref{thm:EEFT} when $E$ is a manifold with boundary $\partial E$ asking that the restriction $\phi|_{\partial E}\colon\partial E\to B$ is also a submersion and applying 
Theorem~\ref{thm:EEFT} to the restriction of $\varphi$ to the interior of $E$ and to the restriction of $\varphi$ to the boundary $\partial E$.

\begin{proof}[Proof of Lemma~\ref{lem:FoF=F}]
Since $g$ is a $C^\ell$-locally trivial fibration, by Theorem~\ref{thm:EEFT} there exists an Ehresmann connection $\mathcal{H}^g$ for $g$, relative
to which every sectionally $C^\ell$-curve in $g(Y)\subset Z$ has horizontal lifts. For any $y\in Y$ we have $T_yY=V_y^g\oplus H_y^g$ where $V_y^g$ and $H_y^g$ are respectively the vertical and 
horizontal subspaces of $T_yY$. Recall that $V_y^g$ is the tangent space of the fibre $g^{-1}(g(y))$ at $y$ and that $H_y^g$ projects isomorphically onto $T_{g(y)}Z$ under $D_yg$.

Analogously, there exists an Ehresmann connection $\mathcal{H}^f$ for $f$, relative
to which every sectionally $C^\ell$-curve in $f(X)\subset Y$ has horizontal lifts. For any $x\in X$ we have $T_xX=V_x^f\oplus H_x^f$ where $V_x^f$ and $H_x^f$ are respectively the vertical and 
horizontal subspaces of $T_xX$. Recall that $V_x^f$ is the tangent space of the fibre $f^{-1}(f(x))$ at $x$ and that $H_x^f$ projects isomorphically onto $T_{f(x)}Y=V_{f(x)}^g\oplus H_{f(x)}^g$ under 
$D_xf$. This isomorphism induces a direct sum decomposition $H_x^f=\tilde{H}_x^f\oplus H_x^{g\circ f}$, where $\tilde{H}_x^f$ and $H_x^{g\circ f}$ correspond respectively to $V_{f(x)}^g$ and 
$H_{f(x)}^g$. Hence we have $T_xX=V_x^f\oplus \tilde{H}_x^f\oplus H_x^{g\circ f}$. Set $V_x^{g\circ f}=V_x^f\oplus \tilde{H}_x^f$, then we have $T_xX=V_x^{g\circ f}\oplus H_x^{g\circ f}$ and we claim 
that $V_x^{g\circ f}$ is the vertical space of $g\circ f$ at $x$ and that the distribution $\mathcal{H}^{g\circ f}=\{H_x^{g\circ f}\}_{x\in X}$ is an Ehresmann connection for $g\circ f$. Firstly, it 
is easy to see that $H_x^{g\circ f}$ is mapped isomorphically onto $T_{g(f(x))}Z$ under $D_x(g\circ f)$
\[
 D_{f(x)}g\bigl(D_xf(H_x^{g\circ f})\bigr)=D_{f(x)}g(H_{f(x)}^g)=T_{g(f(x))}Z.
\]
To see that $V_x^{g\circ f}$ is the vertical space of $g\circ f$ at $x$ we need to check two cases: 1) if $v\in V_x^f$ we have that $D_xf(v)=0$, then $D_{f(x)}g\bigl(D_xf(v)\bigr)=D_{f(x)}g(0)=0$, 2) 
if $v\in \tilde{H}_x^f$ then $D_xf(v)\in V_y^g$ and $D_{f(x)}g\bigl(D_xf(v)\bigr)=0$.

Let $z\in (g\circ f)(X)\subset Z$, $x\in (g\circ f)^{-1}(z)$ and $y=f(x)\in g^{-1}(z)\subset Y$. Let $\alpha\colon I\to Z$ be a sectionally $C^\ell$-curve in $(g\circ f)(X)\subset Z$ with $\alpha(0)=z$ 
and let $\alpha_y\colon I\to f(X)\subset Y$ be its horizontal lift relative to $\mathcal{H}^g$, so we have that $\alpha_y(0)=y$ and $g\circ\alpha_y=\alpha$. Now let $\alpha_x\colon I\to X$ be the 
horizontal lift of $\alpha_y$ relative to $\mathcal{H}^f$, so we have that $\alpha_x(0)=x$ and $f\circ\alpha_x=\alpha_y$. Thus we have $g\circ f\circ\alpha_x=g\circ\alpha_y=\alpha$, so $\alpha_x$ is 
a lift of $\alpha$ by $g\circ f$. To conclude the proof we need to check that $\alpha_x$ is a horizontal lift relative to the Ehresmann connection $\mathcal{H}^{g\circ f}$. Since $\alpha_x\colon I\to 
X$ is the horizontal lift of $\alpha_y$ relative to $\mathcal{H}^f$ we have that $\alpha_x'(t)\in H_{\alpha_x(t)}^f=\tilde{H}_{\alpha_x(t)}^f\oplus H_{\alpha_x(t)}^{g\circ f}$ for every $t\in I$. We 
claim that $\alpha_x'(t)\in H_{\alpha_x(t)}^{g\circ f}$, suppose this is not true, that $\alpha_x'(t)\in\tilde{H}_{\alpha_x(t)}^f$, then $D_{\alpha_x(t)}f(\alpha_x'(t))=\alpha_y'(t)\in 
V_{f(\alpha_x(t))}^g$, but this contradicts the fact that $\alpha_y$ is a horizontal lift of $\alpha$ relative to the Ehresmann connection $\mathcal{H}^g$.
\end{proof}



\end{document}